%% file: moment.tex
\documentclass[12pt]{amsart}
\usepackage{amscd,amsmath,amsthm,amssymb}
\usepackage{mathtools}
\usepackage{mathrsfs}
\usepackage{comment}
\usepackage[all,cmtip]{xy}
\usepackage{xcolor}
\usepackage{enumitem} 
\usepackage{array}
\usepackage{caption}

\usepackage{circuitikz}
\ctikzset{bipoles/resistor/height=0.15}
\ctikzset{bipoles/resistor/width=0.4}

\definecolor{cadmiumgreen}{rgb}{0.0, 0.42, 0.24}
\usepackage[
colorlinks, citecolor=cadmiumgreen,
pagebackref,
pdfauthor={Robin de Jong, Farbod Shokrieh}, 
pdftitle={Tropical moments of tropical Jacobians},
pdfstartview ={FitV},
]{hyperref}

\usepackage[
alphabetic,
backrefs,
msc-links,
nobysame,
lite,
]{amsrefs} 

\usepackage{tikz, float} 
\usetikzlibrary {positioning}

\usetikzlibrary{calc,decorations.markings}
\usetikzlibrary{shapes,snakes}

\usepackage[left=3.5cm,top=3.5cm,right=3.5cm]{geometry}
\def\QQ{{\mathbb Q}}
\def\ZZ{{\mathbb Z}}
\def\RR{{\mathbb R}}
\def\CC{{\mathbb C}}
\def\EE{{\mathbb E}}

\def\T{{\mathbb T}}

\def\Oc{{\mathbf O}}

\def\B{{\mathbf B}}
\def\Q{{\mathbf Q}}
\def\L{{\mathbf L}}
\def\D{{\mathbf D}}
\def\R{{\mathbf R}}
\def\I{{\mathbf I}}
\def\M{{\mathbf M}}
\def\N{{\mathbf N}}
\def\P{{\mathbf P}}
\def\Xib{{\mathbf \Xi}}
\def\F{\mathsf F}
\def\Rrm{\mathrm R}
\def\Prm{\mathrm P}
\def\Cb{{\mathbf C}}
\def\hf{{\mathfrak h}}
\def\Tc{{\mathcal T}}

\def\Ab{{\mathbf A}}
\def\Xb{{\mathbf X}}
\def\Yb{{\mathbf Y}}
\def\Zb{{\mathbf Z}}
\def\Z{{\mathcal Z}}
\def\ones{{\mathbf 1}}
\def\db{\boldsymbol{\delta}}

\def\Gb{{\mathbf G}}

\def\vb{{\mathbf v}}

\def\eb{\mathrm{Par}(e)}

\def\st{{\mathfrak s}}
\def\tf{{\mathfrak t}}
\def\d{\mathrm{d}}
\def\h{\mathrm{h}}

\def\Oc{{\mathcal O}}

\def\Mc{{\mathcal M}}

\def\MA{{\mathcal A}}

\def\Fc{{\mathcal F}}

\def\Sc{{\mathcal S}}

\def\C{{\mathcal C}}

\def\Dc{{\mathcal D}}

\def\opn#1#2{\def#1{\operatorname{#2}}} 
\opn\depth{depth} 
\opn\codim{codim}
\opn\ini{in} 
\opn\LM{LM}
\opn\LC{LC}
\opn\NF{NF}
\opn\Merge{Merge}
\opn\sgn{sgn}
\opn\div{div} 
\opn\Div{Div} 
\opn\Pic{Pic}
\opn\Prin{Prin}
\opn\Del{Del}
\opn\op{op}
\opn\ends{ends}
\opn\indeg{indeg} 
\opn\sign{sign}
\opn\outdeg{outdeg}
\opn\red{red}
\opn\Spec{Spec} 
\opn\Supp{Supp} 
\opn\supp{supp} 
\opn\Ker{Ker} 
\opn\Coker{Coker} 
\opn\Hom{Hom}
\opn\Tor{Tor} 
\opn\id{id}
\opn\span{span}
\opn\Image{Image}
\opn\con{conv} 
\opn\relint{rel.int} 
\opn\vol{vol}
\opn\val{val}
\opn\Zh{Zh}
\opn\Ber{Ber}
\opn\Vor{Vor}
\opn\Vol{Vol}
\opn\Covol{Covol}
\opn\Jac{Jac}
\opn\Tr{Trace}
\opn\Dir{Dir}
\opn\can{can}
\opn\syz{{\rm syz}}
\opn\spoly{{\rm spoly}}
\opn\LM{{\rm LM}}
\opn\lm{{\rm lm}}
\opn\lcm{{\rm lcm}} 
\opn\Ac{\mathcal A}
\opn\dist{dist}
\opn\pd{pd}
\opn\en{en}
\opn\PL{PL}
\opn\dmeasz{DMeas_0}
\opn\dmeas{DMeas}
\opn\TP{T}
\opn\circu{circ}
\opn\cocirc{cocirc}
\opn\Proj{Proj}

\def\Implies{\ifmmode\Longrightarrow \else
        \unskip${}\Longrightarrow{}$\ignorespaces\fi}
\def\implies{\ifmmode\Rightarrow \else
        \unskip${}\Rightarrow{}$\ignorespaces\fi}
\def\iff{\ifmmode\Longleftrightarrow \else
        \unskip${}\Longleftrightarrow{}$\ignorespaces\fi}

\newtheorem{Theorem}{Theorem}[section]
\newtheorem{Lemma}[Theorem]{Lemma}

\newtheorem{Proposition}[Theorem]{Proposition}

\theoremstyle{remark}
\newtheorem{Remark}[Theorem]{Remark}

\theoremstyle{definition}
\newtheorem{Example}[Theorem]{Example}
\newtheorem{Definition}[Theorem]{Definition}
\newtheorem*{Notation}{Notation}

\def\qed{\ifhmode\textqed\fi
      \ifmmode\ifinner\quad\qedsymbol\else\dispqed\fi\fi}
\def\textqed{\unskip\nobreak\penalty50
       \hskip2em\hbox{}\nobreak\hfil\qedsymbol
       \parfillskip=0pt \finalhyphendemerits=0}
\def\dispqed{\rlap{\qquad\qedsymbol}}

\numberwithin{equation}{section}

\tikzstyle{Cwhite}=[scale = .8,circle, fill = white, minimum size=3mm] 
\tikzstyle{Cgray}=[scale = .4,circle, fill = gray, minimum size=3mm] 
\tikzstyle{Cblack2}=[scale = .4,circle, fill = black, minimum size=5mm] 
\tikzstyle{Cblack}=[scale = .7,circle, fill = black, minimum size=3mm]
\tikzstyle{C0}=[scale = .9,circle, fill = black!0, inner sep = 0pt, minimum size=3mm]
\tikzstyle{C1}=[scale = .7,circle, fill = black!0, inner sep = 0pt, minimum size=3mm]
\tikzstyle{Cred}=[scale = .4,circle, fill = red, minimum size=3mm]

\def\h{\mathrm{h}}
\newcommand*\isom{\xrightarrow{\sim}}

\begin{document}

\title{Tropical moments of tropical Jacobians}

\author{Robin de Jong}
\email{\href{mailto:rdejong@math.leidenuniv.nl}{rdejong@math.leidenuniv.nl}}

\author {Farbod Shokrieh}
\email{\href{mailto:farbod@math.cornell.edu}{farbod@math.cornell.edu}, \href{mailto:farbod@math.ku.dk}{farbod@math.ku.dk}}


\subjclass[2010]{
\href{https://mathscinet.ams.org/msc/msc2010.html?t=14T05}{14T05},
\href{https://mathscinet.ams.org/msc/msc2010.html?t=14G40}{14G40},
\href{https://mathscinet.ams.org/msc/msc2010.html?t=11G50}{11G50},
\href{https://mathscinet.ams.org/msc/msc2010.html?t=94C05}{94C05},
\href{https://mathscinet.ams.org/msc/msc2010.html?t=52C45}{52C45}}


\date{\today}

\begin{abstract}
Each metric graph has canonically associated to it a polarized real torus called its tropical Jacobian. A fundamental real-valued invariant associated to each polarized real torus is its tropical moment.  We give an explicit and efficiently computable formula for the tropical moment of a tropical Jacobian in terms of potential theory on the underlying metric graph. We show that there exists a universal linear relation between the tropical moment, the tau invariant, and the total length of a metric graph. We argue that this linear relation is a non-archimedean analogue of a recent remarkable identity established by Wilms for invariants of compact Riemann surfaces. We also relate our work to the computation of heights attached to principally polarized abelian varieties.
\end{abstract}

\maketitle

\setcounter{tocdepth}{1}
\tableofcontents 

\section{Introduction} \label{sec:Intro}
\renewcommand*{\theTheorem}{\Alph{Theorem}}
\subsection{Background and aim}

A {\em metric graph} is a compact, connected length metric space $\Gamma$ homeomorphic to a topological graph. It is known that metric graphs are the appropriate analogues of compact, connected Riemann surfaces in the non-archimedean setting. For example, one can associate to $\Gamma$ its {\em tropical Jacobian} $\Jac(\Gamma) = H_1(\Gamma, \RR) / H_1(\Gamma, \ZZ)$, which is a real torus endowed with a canonical inner product $[\cdot,\cdot]$ on its tangent space (i.e.,\ a {\em polarized real torus}, or {\em polarized tropical abelian variety}). 

The lattice $H_1(\Gamma,\ZZ)$ together with the inner product $[\cdot,\cdot]$ canonically determines a compact convex region in $H_1(\Gamma,\RR)$ given as
\[ \Vor(0) =\left\{ z \in H_1(\Gamma,\RR) \, \colon \, [ z  , z ] \leq   [ z-\lambda,z-\lambda ] \,\, \textrm{for all} \,\, \lambda \in H_1(\Gamma,\ZZ) \right\} \, , \]
known as the {\em Voronoi polytope} centered at the origin. Let $\mu_L$ denote the Lebesgue measure on $H_1(\Gamma,\RR)$, normalized to give $\Vor(0)$ unit volume. In this paper we are interested in the {\em tropical moment} of $\Jac(\Gamma)$,  defined as the integral

\begin{equation} \label{eqn:integral_expr} I(\Jac(\Gamma)) = \int_{\Vor(0)} [ z,z ] \, \d \mu_L(z) \, . 
\end{equation}
We provide an explicit and efficiently computable formula for $I(\Jac(\Gamma))$ in terms of potential theory on $\Gamma$ (see Theorem~\ref{thm:momentIntro}). As we will explain, our formula seems rather unexpected from the computational complexity point of view.

Next, we establish a simple connection between $I(\Jac(\Gamma))$ and a well known potential theoretic invariant associated to $\Gamma$ called the {\em tau invariant}, denoted by $\tau(\Gamma)$ (see Theorem~\ref{thm:linearIntro}). The invariant $\tau(\Gamma)$ can be traced back to the fundamental work of Chinburg and Rumely \cite{cr} in their study of the Arakelov geometry of arithmetic surfaces at non-archimedean places. 

Finally, we argue that the identity we prove in Theorem~\ref{thm:linearIntro} is an analogue, in the non-archimedean setting, of a recent remarkable identity established by Wilms \cite{Wilms} relating three fundamental invariants of compact connected Riemann surfaces (see \S\ref{sec:wilms}). We also explain how our work is connected with the computation of the stable Faltings height of principally polarized abelian varieties defined over number fields (see \S\ref{sec:connection}).

\subsection{Formula for the tropical moment}

It is well known that one can think of a metric graph $\Gamma$ as an electrical network. For $x,y,z \in \Gamma$ we define $j_z(x,y)$ to be the electric potential at $x$ if one unit of current enters the network at $y$ and exits at $z$, with $z$ `grounded' (i.e. has zero potential). The {\em $j$-function} provides a fundamental solution of the Laplacian operator on $\Gamma$, so naturally it is an important function in studying harmonic analysis on $\Gamma$. The {\em effective resistance} between $x$ and $y$ is defined as $r(x,y) = j_y(x,x)$, which has the expected physical meaning in terms of electrical networks. See \S\ref{sec:laplace}.

It is convenient to fix a vertex set (which is a finite non-empty set containing all the branch points of $\Gamma$) and to think of the metric graph $\Gamma$ as being obtained from a finite (combinatorial) graph $G$. The metric data will be encoded by positive real numbers $\ell(e)$ for each edge $e$ of $G$. We call the finite weighted graph arising in this way a {\em model} of $\Gamma$. If $e$ is an edge of $G$, the {\em Foster coefficient} of $e$ is defined by $\F(e) = 1-{r(u, v)}/{\ell(e)}$, where $u$ and $v$ are the two endpoints of the edge $e = \{u,v\}$ (see Definition~\ref{def:foster}). 

Our formula for the tropical moment of a tropical Jacobian is as follows. 

\begin{Theorem} \label{thm:momentIntro} (=Theorem~\ref{thm:momentformula})
Let $\Gamma$ be a metric graph. Fix a model $G$ of $\Gamma$ and fix a vertex $q$ of $G$. Then
\[
I(\Jac(\Gamma)) = \frac{1}{12}\sum_{e} {\F(e)^2 \ell(e) +  \frac{1}{4}\sum_{e = \{u,v\}} \left(r(u,v) - \frac{j_{u}( v, q)^2 + j_{v}( u, q)^2}{\ell(e)}\right)} \, ,
\]
where the sums are over all edges $e \in E(G)$.
\end{Theorem}
The formula in Theorem~\ref{thm:momentIntro} seems to be rather unexpected from the point of view of computational complexity. As follows from our discussion in \S\ref{sec:pot}, in order to obtain all ingredients of the formula, one essentially only needs to perform some basic operations on the discrete Laplacian matrix of $G$. 
It follows that, given a model $G$ of $\Gamma$, one can compute $I(\Jac(\Gamma))$ in time $O(n^\omega)$, where $n$ is the number of vertices of $G$, and where $\omega$ is the exponent for the `matrix multiplication' algorithm (currently $\omega < 2.38$). We refer to Remark~\ref{rmk:I}~(ii) for details. 

On the other hand, computing tropical moments for general polarized real tori is expected to be NP-complete (or worse) in the hierarchy of complexity classes, and there seems to be no {\em a priori} reason to expect that the computation might be easier for tori arising from metric graphs. For example,  if one wants to compute the tropical moment via the `simplex method' (see, e.g., \cite[Chapter 21, \S2]{cs}), then one needs to know the vertices of $\Vor(0)$. However, as is shown in \cite{ssv}, even computing the {\em number} of vertices of $\Vor(0)$ is already $\#$P-hard.

The proof of Theorem~\ref{thm:momentIntro} is very subtle. Our strategy is as follows. To handle the integral in (\ref{eqn:integral_expr}), we provide an explicit polytopal decomposition of $\Vor(0)$. Given a base point $q\in V(G)$, there is a full-dimensional polytope $\sigma_T + \C_T$ in our decomposition attached to each {\em spanning tree} $T$ of $G$. Here $\C_T$ is a centrally symmetric polytope, which makes $\sigma_T$ into the center of $\sigma_T + \C_T$ (see Theorem~\ref{thm:vordecomp}). The desired integral over $\Vor(0)$ is then a sum of integrals over each $\sigma_T + \C_T$. 

The contribution from the centrally symmetric polytopes $\C_T$ is rather easy to handle. The real difficulty comes in handling the contributions from the centers $\sigma_T$. We introduce the notion of {\em energy level} of rooted spanning trees (see Definition~\ref{def:enlev}). A crucial ingredient is the notion of {\em cross ratio} introduced in \cite{FR2} for electrical networks. We prove that the weighted average of energy levels over all spanning trees has a remarkably simple expression in terms of values of the $j$-function (see Theorem~\ref{thm:energylevel}). Proving this, in turn, uses some subtle computations related to functions arising from {\em random spanning trees} (see \S\ref{sec:random}). We also repeatedly use our generalized (and quantitative) version of Rayleigh's law in electrical networks, as developed in the companion paper \cite{FR2}.

\subsection{Tropical moment and tau invariant}
Fix a point $q \in \Gamma$. Let $f(x) = \frac{1}{2} r(x, q)$, where $r$ denotes the effective resistance on $\Gamma$ as above. The tau invariant of $\Gamma$ can be given by the simple formula  (see Definition  \ref{def:tau})
\[ \tau(\Gamma)= \int_{\Gamma}{\left(f'(x)\right)^2 \d x} \, .
\]
Here $\d x$ denotes the (piecewise) Lebesgue measure on $\Gamma$. We will prove the following remarkably simple linear identity connecting $\tau(\Gamma)$ and $I(\Jac(\Gamma))$. We refer to Definition \ref{def:total_length} for the {\em total length} of $\Gamma$.
\begin{Theorem}\label{thm:linearIntro} (=Theorem~\ref{thm:linear})
Let $\Gamma$ be a metric graph. Let $\tau(\Gamma)$ denote the tau invariant of~$\Gamma$, and let $\ell(\Gamma)$ denote its total length. Then the identity
\[ \frac{1}{2}\tau(\Gamma)+I(\Jac(\Gamma)) = \frac{1}{8} \ell(\Gamma) \]
holds in $\RR$.
\end{Theorem}
 In the next two sections, we would like to argue that the linear relation in Theorem~\ref{thm:linearIntro} is a non-archimedean analogue of a recent remarkable identity established by Wilms \cite{Wilms} for invariants of compact Riemann surfaces. 
 
\subsection{Tropical moment as a tropical limit} \label{sec:trop_limit}

A polarized real torus is a real torus equipped with a Euclidean inner product on its tangent space. For each given polarized real torus $\T$ one clearly has an associated Voronoi polytope $\Vor(0)$ centered at the origin, as well as an associated tropical moment $I(\T)$. Following \cite{mz, FRSS} one can also naturally attach to $\T$ a {\em tropical Riemann theta function} $\varPsi(z)$, which is a real-valued piecewise affine function on the tangent space of $\T$ (see \S\ref{sec:realtori}). By modifying $\varPsi$ into the function $\|\varPsi\|(z)=\varPsi(z) + \frac{1}{2}[z,z]$ one obtains a lattice-invariant function, which therefore descends to give a map $\|\varPsi\| \colon \T \rightarrow \RR$.  It is not difficult to see that the tropical moment of $\T$ can alternatively be written as
\[ I(\T) = \int_{\T} 2 \,\| \varPsi \| \, \d \mu_{H}  \, , \]
where $\mu_H$ denotes the Haar measure on $\T$, normalized to give $\T$ unit volume. 

Now let $A$ be a complex abelian variety, endowed with a principal polarization $\lambda \colon A \isom A^t$. Let $\|\vartheta\|$ denote the normalized version of the {\em complex Riemann theta function} on the compact complex torus $A(\CC)$, as defined in \cite[page~401]{faltCalc}. Let $\mu$ denote the Haar measure on $A(\CC)$, normalized to give $A(\CC)$ unit volume. The {\em $I$-invariant}  of the pair $(A,\lambda)$, introduced by Autissier \cite{aut}, is the integral 
\begin{equation} \label{def_arch_lambda_bis}
I(A,\lambda) = -\int_{A(\CC)} \log \|\vartheta\|  \, \d \, \mu - \frac{g}{4} \log 2 \, .  
\end{equation}
Here and below we denote by $\log$ the natural logarithm, and we set $g=\dim(A)$. 

In \cite{FRSS} the  tropical Riemann theta function $\varPsi$  is obtained as the {\em tropicalization} of a {\em non-archimedean Riemann theta function}. This suggests that one may see $\varPsi$ as the {\em `tropical limit'} of the complex Riemann theta function, and  the modified tropical Riemann theta function $\|\varPsi\|$ as the `tropical limit' of the normalized complex Riemann theta function $\|\vartheta\|$. This suggests, in turn, that the tropical moment of a polarized real torus may be seen as the `tropical limit' of (two times) the $I$-invariant of a principally polarized complex abelian variety.

\subsection{Connection with an identity of Wilms} \label{sec:wilms}

Let $C$ be a compact and connected Riemann surface of genus $g\geq 1$. Let $\delta_F(C)$ be the {\em Faltings delta invariant} of $C$ (see \cite[page 402]{faltCalc}), let $\varphi(C)$ be the {\em Zhang--Kawazumi invariant} of $C$ (see \cite[page 6]{ZhGrSch}), and let $I(\Jac(C))$ be the $I$-invariant of the Jacobian $\Jac(C)$ of $C$, seen as a principally polarized complex abelian variety. Put $\kappa_0=\log(\pi \sqrt{2})$.  Wilms recently established the following remarkable identity.

\vspace{2mm}
\noindent {\bf Theorem.} (\cite[Theorem~1.1]{Wilms})
The identity 
\begin{equation} \label{eq:wilmsintro}
\delta_F(C)  - 2 \, \varphi(C) = 12 \, \left(2 \, I(\Jac(C))\right) + a(g)
\end{equation}
holds in $\RR$, where $a(g) = \left(4\log(2\pi)- 12\,\kappa_0 \right)g$ is a constant depending only on $g$.
\vspace{2mm}

A {\em polarized} metric graph $\overline{\Gamma}$ (or an {\em abstract tropical curve}) is a pair $(\Gamma, \mathbf{q})$ consisting of a metric graph $\Gamma$ and a labeling $\mathbf{q} \colon V(G) \to \ZZ_{\geq 0}$ of the vertices in a model $G$ of $\Gamma$. Let $\varepsilon(\overline{\Gamma})$ denote the $\varepsilon$-invariant of $\overline{\Gamma}$ (see  \cite{Zhang93, Moriwaki}), and let $\varphi(\overline{\Gamma})$ denote the $\varphi$-invariant of $\overline{\Gamma}$ (see  \cite[page~7]{ZhGrSch}). 
The main result of \cite{deJongFaltings} shows that the tropical invariant $\ell(\Gamma) + \varepsilon(\overline{\Gamma})$ can be seen as the `tropical limit' of the delta invariant $\delta_F(C)$, and the (partial) asymptotic results in \cite{deJongAsymptotic, deJongPointlike} indicate that the tropical invariant $\varphi(\overline{\Gamma})$ can be seen 
as the `tropical limit' of $\varphi(C)$. In \S\ref{sec:trop_limit} we argued that the tropical moment should be seen as the `tropical limit' of two times the $I$-invariant.

Therefore, a non-archimedean analogue of Wilms's identity \eqref{eq:wilmsintro} should have the form
 \begin{equation}\label{eq:nonarch_wilmsintro}
 \left(\ell(\Gamma) + \varepsilon(\overline{\Gamma})\right) - 2 \,\varphi(\overline{\Gamma}) = 12 \, I(\Jac(\Gamma))  \, .
 \end{equation}
It is shown in \cite[Proposition~9.2]{djneron} that the identity
\begin{equation}
 \left( \ell(\Gamma) +\varepsilon(\overline{\Gamma}) \right) - 2\,\varphi(\overline{\Gamma}) = 12\,\left(\frac{1}{8} \ell(\Gamma) - \frac{1}{2} \tau(\Gamma)\right) 
\end{equation}
holds. We see that \eqref{eq:nonarch_wilmsintro} holds if and only if $\frac{1}{8} \ell(\Gamma) - \frac{1}{2} \tau(\Gamma) = I(\Jac(\Gamma))$. In other words, our Theorem~\ref{thm:linearIntro} precisely establishes a non-archimedean analogue of Wilms's identity. 

\subsection{Connection with arithmetic geometry} \label{sec:connection}

As is well known, a metric graph may be canonically interpreted as a {\em skeleton} of a {\em Berkovich curve}, and the tropical Jacobian of a metric graph can be interpreted as the {\em canonical skeleton} of a {\em Berkovich Jacobian variety} \cite{brab}.  More generally, one may view the canonical skeleton of any Berkovich polarized abelian variety as a polarized real torus in a canonical way, via {\em non-archimedean uniformization}. Based on this connection, our results can be applied in the study of polarized abelian varieties. For example, we have the following application concerning the computation of {\em Arakelov heights} attached to principally polarized abelian varieties defined over a number field. For more background and for terminology used in this section we refer to \cite{FR1}.

Let $k$ be a number field, and let $M(k)_0$ and $M(k)_\infty$ denote the set of non-archimedean places and the set of complex embeddings of $k$.  Let $(A,\lambda)$ be a principally polarized abelian variety defined over $k$. Assume that $A$ has semistable reduction over $k$.  For $v \in M(k)_\infty$ we let $I(A_v,\lambda_v)$ denote the $I$-invariant of the principally polarized complex abelian variety $(A_v,\lambda_v)$ obtained by extending scalars to $\bar{k}_v \simeq \CC$.  For $v \in M(k)_0$ we let $I(A_v,\lambda_v)$ denote the tropical moment of the canonical skeleton of the Berkovich analytification of $(A,\lambda)$ at $v$, viewed as a polarized real torus. Let $Nv$ be the cardinality of the residue field at $v \in M(k)_0$. 

\vspace{2mm}
\noindent {\bf Theorem.} (\cite[Theorem~A]{FR1})
 Let $\varTheta$ be a symmetric effective divisor on $A$ that defines the polarization $\lambda$, and put $L=\mathcal{O}_A(\varTheta)$. Let $\h'_L(\varTheta)$ denote the {\em N\'eron--Tate height} of the cycle $\varTheta$, and let $\h_F(A)$ denote the {\em stable Faltings height} of $A$. Set $g=\dim(A)$. Then the equality
\begin{equation} \label{two_heights}
 \h_F(A) =  2g \, \h'_L(\varTheta) -   \kappa_0 \, g + \frac{1}{[k:\QQ]} \left( \sum_{v \in M(k)_0} I(A_v,\lambda_v) \log Nv + 2 \sum_{v \in M(k)_\infty} I(A_v,\lambda_v)   \right) 
 \end{equation}
holds in $\RR$.
\vspace{2mm}

Assume that $v \in M(k)_0$ is a finite place such that the canonical skeleton of the Berkovich analytification of $(A,\lambda)$ at $v$ can be realized as the tropical Jacobian of some (explicitly given) metric graph. For example $(A,\lambda)$ could be the Jacobian variety of a smooth projective geometrically connected curve  with semistable reduction over $k$. Then Theorem~\ref{thm:momentIntro} or Theorem~\ref{thm:linearIntro} can be applied to compute the local term  $I(A_v,\lambda_v)$ efficiently. We shall illustrate this in \S\ref{sec:arakelov} by discussing the case of Jacobian varieties of dimension two in some detail. We mention that by combining (\ref{two_heights}) with Theorem \ref{thm:linearIntro} one recovers \cite[Theorem~1.6]{djneron}.

\subsection{Structure of the paper}

In \S\ref{sec:realtori}, we review the notion of polarized real tori and define the notion of tropical moments. 
In \S\ref{sec:metric}, we review the notions of weighted graphs and of metric graphs and their models. Also we introduce the tropical Jacobian of a metric graph.
In \S\ref{sec:pot}, we review potential theory and harmonic analysis on metric graphs, mainly from the perspective of our companion paper \cite{FR2}.
In \S\ref{sec:random}, we study two functions that arise from the theory of random spanning trees.  
In \S\ref{sec:crossavg}, we introduce the notion of energy levels of rooted spanning trees, and prove that the average of energy levels has a simple expression in terms of the $j$-function.
In \S\ref{sec:comb}, we study the combinatorics of the Voronoi polytopes arising from graphs, and present our suitable polytopal decomposition. 
In \S\ref{sec:tropmoment}, we prove Theorem~\ref{thm:momentIntro}.
In \S\ref{sec:arakelov}, we elaborate upon (\ref{two_heights}) in the case of a Jacobian variety of dimension two. 
In \S\ref{sec:tau}, we introduce the tau invariant and prove a formula that allows to compute the tau invariant efficiently.
In \S\ref{sec:linearrel}, we prove Theorem~\ref{thm:linearIntro}.

\renewcommand*{\theTheorem}{\arabic{section}.\arabic{Theorem}}

\section{Polarized real tori and tropical moments}
\label{sec:realtori}

The purpose of this section is to set notations and terminology related to polarized real tori and their tropical moments, and to discuss the connection with the tropical Riemann theta function as studied in \cite{mz, FRSS}.

\subsection{Polarized real tori} \label{sec:poltori}
A {\em (Euclidean) lattice} is a pair $(\Lambda, [ \cdot,\cdot  ])$ consisting of a finitely generated free $\ZZ$-module $\Lambda$ and a  symmetric bilinear form $[ \cdot,\cdot  ] \colon \Lambda \times \Lambda \to \RR$ such that the induced symmetric bilinear form on $\Lambda_\RR=\Lambda \otimes_\ZZ \RR$ (which we likewise denote by $[\cdot,\cdot]$) is positive definite. Attached to each lattice $(\Lambda,[\cdot,\cdot])$ one has a real torus $\T=\Lambda_\RR/\Lambda$, equipped with a natural structure of compact Riemannian manifold. We refer to the Riemannian manifold $\T$ as a {\em polarized real torus}. The tropical Jacobian of a metric graph (see Section \S\ref{sec:trop_Jac}) is an example of a polarized real torus.

\subsection{Voronoi decompositions and tropical moment} \label{sec:voronoi}

Let $\T$ be a polarized real torus coming from a lattice $(\Lambda,[\cdot,\cdot])$ as above.  For each $\lambda \in \Lambda$ we denote by $\Vor( \lambda)$ the Voronoi polytope of the lattice $(\Lambda,[ \cdot,\cdot  ])$ around $\lambda$:
\[ \Vor(\lambda) \coloneqq \{ z \in \Lambda_{\RR} \, \colon  \, [ z - \lambda , z -\lambda ] \leq [ z-\lambda',z-\lambda' ] \  \textrm{for all} \, \lambda' \in \Lambda \} \, . \]
Note that, for each $\lambda \in \Lambda$, we have $\Vor(\lambda) = \Vor(0) + \lambda$. Moreover $\Vor(0)$, up to some identifications on its boundary, is a fundamental domain for the translation action of $\Lambda$ on $\Lambda_{\RR}$.
\begin{Definition} The {\em tropical moment} of the polarized real torus $\T$ is set to be the value of the integral
\begin{equation} \label{normalized_second}
I(\T)  \coloneqq \,  \int_{\Vor(0)} [ z,z ] \, \d \mu_L(z) \, ,  
\end{equation}
where $\mu_L$ is the Lebesgue measure on $\Lambda_{\RR}$, normalized to have $\mu_L(\Vor(0)) = 1$.
\end{Definition}
\begin{Remark} The tropical moment of $\T$ is exactly the {\em normalized second moment} of the lattice $(\Lambda,[\cdot,\cdot])$ as studied in \cite[Chapter 21]{cs}.
\end{Remark}

\subsection{Tropical Riemann theta function} \label{sec:thetafunc}
Let $\T$ be a polarized real torus coming from a lattice $(\Lambda,[\cdot,\cdot])$ as above. 
We define the \emph{tropical Riemann theta function} of  $\T$ to be the function $\varPsi \colon \Lambda_\RR \rightarrow \RR$ given by (see \cite{mz, FRSS}) 
\[ \varPsi(z) \coloneqq \min_{\lambda \in \Lambda} \left\{  [ z, \lambda ] + \frac{1}{2}[ \lambda,\lambda ] \right\} \]
for $z \in \Lambda_\RR$. 
 As is easily checked we have a functional equation
\[ \varPsi(z) = \varPsi(z+\mu) + [ z , \mu ] + \frac{1}{2} [ \mu,\mu ] \]
for all $z \in \Lambda_{\RR}$ and $\mu \in \Lambda$. The function $\varPsi$ is piecewise affine on $\Lambda_{\RR}$.  The modified theta function
\begin{equation} \label{eqn:modified} \|\varPsi\|(z) \coloneqq \varPsi(z) + \frac{1}{2}[ z,z ] 
\end{equation}
on $\Lambda_{\RR}$ is $\Lambda$-invariant and hence descends to $\T$. Explicitly, we have
\[ 2\,\|\varPsi\|(z) =  \min_{\lambda \in \Lambda}  \, [ z-\lambda,z-\lambda ] \]
for all $z \in \Lambda_{\RR}$. As is easy to check, the tropical moment (\ref{normalized_second}) has a simple expression in terms of the modified theta function (\ref{eqn:modified}), namely 
\begin{equation} \label{eq:tropmoment} I(\T)  = \int_{\T} 2 \,\| \varPsi \| \, \d \mu_{H}  \, . 
\end{equation}
Here $\mu_H$ denotes the Haar measure on the compact topological group $\T$, normalized to have $\mu_H(\T) =1$. 
\begin{Remark}
The corner locus of the tropical Riemann theta function $\varPsi$ (i.e. the lift of the {\em tropical theta divisor} on $\T$ to its tangent space  $\Lambda_{\RR}$) consists precisely of the boundaries of the Voronoi polytopes $\Vor(\lambda)$.
\end{Remark}

\section{Metric graphs, models, and tropical Jacobians}
\label{sec:metric}

The purpose of this section is to set notations and terminology related to weighted graphs, metric graphs, and their models. We also define the tropical Jacobian of a metric graph (see \S\ref{sec:trop_Jac}). Most of the material in this section is straightforward, and we leave details to the interested reader.

\subsection{Weighted graphs}\label{sec:weightedgraph}
By a {\em weighted graph} we mean a finite weighted connected multigraph $G$ with no loop edges. The set of vertices of $G$ is denoted by $V(G)$ and the set of edges of $G$ is denoted by $E(G)$. We let $n = |V (G)|$ and $m = |E(G)|$. An edge $e$ is called a {\em bridge} if $G\backslash e$ is disconnected. The weights of edges are determined by a {\em length} function
$ \ell \colon E(G) \to \RR_{>0}$.
We let $\mathbb{E}(G) = \{e, \bar{e} \colon e \in E(G)\}$ denote the set of {\em oriented edges}. We have $\bar{\bar{e}} = e$. For each subset $\MA \subseteq \EE(G)$, we define $\overline{\MA} = \{\bar{e} \colon e \in \MA\} $. An {\em orientation} $\Oc$ on $G$ is a partition $\mathbb{E}(G) = \Oc \cup \overline{\Oc}$. We have an obvious extension of the length function 
$\ell \colon \mathbb{E}(G) \rightarrow  \RR_{>0}$
by requiring $\ell(e) = \ell(\bar{e})$. There is a natural map
$\mathbb{E}(G) \rightarrow V(G) \times V(G)$ sending an oriented edge $e$ to $(e^+, e^-)$, where $e^- $ is the start point of $e$ and $e^+$ is the end point of $e$.

\begin{Notation}
For $e \in E(G)$ we sometimes refer to its
endpoints by $e^+,e^-$ even when an orientation is not fixed, so $e = \{e^+, e^-\}$. We only allow ourselves to do this if the underlying expression is symmetric with respect to $e^+$ and $e^-$, so there is no danger of confusion. The reader is welcome to fix an orientation $\Oc$ and think of $e^+$ and $e^-$ in the sense explained above.
\end{Notation}

A {\em spanning tree} $T$ of $G$ is a maximal subset of $E(G)$ that contains no circuit (closed simple path). Equivalently, $T$ is a minimal subset of $E(G)$ that connects all vertices of $G$. 

For a fixed $q \in V(G)$ and spanning tree $T$ of $G$ we will refer to the pair $(T, q)$ as a {\em spanning tree with a root at $q$} (or just a {\em rooted spanning tree}). The choice of $q$ imposes a preferred orientation on all edges of $T$. Namely, one can require that all edges are oriented {\em away from $q$} on the spanning tree $T$ (see Figure~\ref{fig:graphs}~(c)). We denote this orientation on $T$ by $\Tc_q \subseteq \mathbb{E}(G)$. 

Given a commutative ring $R$, it is convenient to define the $1$-chains with coefficients in $R$ as the free module

\[
C_1(G,R) \coloneqq \frac{\bigoplus_{e \in \mathbb{E}(G)} R e }{ \langle e+\bar{e} \colon e \in \Oc \rangle}  \,.
\]
Note that $\bar{e} = -e$ in $C_1(G,R)$. For each orientation $\Oc$ on $G$ we have an isomorphism $ C_1(G,R) \simeq \bigoplus_{e \in \Oc}  R e$. 
For each subset $\MA \subseteq \EE(G)$, we define its {\em associated $1$-chain} as 
$\boldsymbol{\gamma}_\MA = \sum_{e \in \MA} e$.

\subsection{Metric graphs and models} 
 A {\em metric graph} (or {\em metrized graph}) is a pair $(\Gamma,d)$ consisting of a compact connected topological graph $\Gamma$, together with an {\em inner metric} $d$. 
Equivalently, if $\Gamma$ is not a one-point space, then a metric graph is a compact connected metric space $\Gamma$ which has the property that every point has an open neighborhood isometric to a star-shaped set, endowed with the path metric.

The points of $\Gamma$ that have valency different from $2$ are called {\em branch points} of $\Gamma$. A {\em vertex set} for $\Gamma$ is a finite set $V$ of points of $\Gamma$ containing all the branch points of $\Gamma$ with the property that for each connected component $c$ of $\Gamma \setminus V$, the closure of $c$ in $\Gamma$ is isometric with a closed interval. 

A vertex set $V$ for $\Gamma$ naturally determines a weighted graph $G$  by setting $V(G)=V$, and by setting $E(G)$ to be the set of connected components of $\Gamma \setminus V$. We call $G$ a {\em model} of $\Gamma$. An {\em edge segment} (based on the choice of a vertex set $V$) is the closure in $\Gamma$ of a connected component of $\Gamma \setminus V$. Note that there is a natural bijective correspondence between $E(G)$ and the edge segments of $\Gamma$ determined by $V$. By a small abuse of terminology we will refer to the elements of $E(G)$ also as edge segments of $\Gamma$. Given an edge segment $e\subset \Gamma$ (based on the choice of a vertex set $V$) we denote its boundary $\partial e \subset V$ by $\partial e =\{e^-,e^+\}$. In particular we will also use the notation $\{e^-,e^+\}$ for the boundary set of an edge segment $e$ if there is no (preferred) orientation present. We hope that this does not lead to confusion. 

Conversely, every weighted graph $G$ naturally determines a metric graph $\Gamma_G$ containing $V(G)$ by glueing closed intervals $[0,\ell(e)]$ for $e \in E(G)$ according to the incidence relations. Note that $V(G)$ is naturally a vertex set of $\Gamma_G$, and the associated model is precisely $G$. See Figure~\ref{fig:graphs}~(a) and (b).

\begin{figure}[h!]
$$
\begin{xy}
(0,0)*+{
	\scalebox{.7}{$
	\begin{tikzpicture}
	\draw[black, ultra thick, -] (0,1.2) to [out=-45,in=90] (.6,-.05);
	\draw[black, ultra thick] (.6,0) to [out=-90,in=45] (0,-1.2);
	\draw[black, ultra thick, -] (0,1.2) to [out=-135,in=90] (-.6,-.05);
	\draw[black, ultra thick] (-.6,0) to [out=-90,in=135] (0,-1.2);
	\draw[black, ultra thick, -] (0,1.2) -- (0,0);
	\draw[black, ultra thick] (0,0.1) -- (0,-1.2);
	\end{tikzpicture}
	$}
};
(-6.5,0)*+{\mbox{{\smaller $4$}}};
(2,-1)*+{\mbox{{\smaller $1$}}};
(6.75,-0.05)*+{\mbox{{\smaller $4$}}};
(0,-15)*+{(a)};
\end{xy}
\ \ \ \ \ \ \ \ \ \ 
\begin{xy}
(0,0)*+{
	\scalebox{.7}{$
	\begin{tikzpicture}
	\draw[black, ultra thick] (-1.2,0) to (-.6,.6);
	\draw[black, ultra thick] (-.6,.6) to (0,1.2);
	\draw[black, ultra thick] (-1.2,0) to (0,0);
	\draw[black, ultra thick] (0,0) to (1.2,0);
	\draw[black, ultra thick] (-1.2,0) to (-.6,-.6);
	\draw[black, ultra thick] (-.6,-.6) to (0,-1.2);
	\draw[black, ultra thick] (0,1.2) to (.6,.6);
	\draw[black, ultra thick] (.6,.6) to (1.2,0);
	\draw[black, ultra thick] (1.2,0) to (.6,-.6);
	\draw[black, ultra thick] (.6,-.6) to (0,-1.2);
	\fill[black] (0,1.2) circle (.1);
	\fill[black] (0,-1.2) circle (.1);
	\fill[black] (1.2,0) circle (.1);
	\fill[black] (-1.2,0) circle (.1);
	\end{tikzpicture}
	$}
};
(-6.5,7)*+{\mbox{{\smaller $2$}}};
(0,-2)*+{\mbox{{\smaller $1$}}};
(6.5,7)*+{\mbox{{\smaller $2$}}};
(-6.5,-7)*+{\mbox{{\smaller $2$}}};
(6.5,-7)*+{\mbox{{\smaller $2$}}};
(0,-15)*+{(b)};
\end{xy}
\ \ \ \ \ \ \ \ \ \ 
\begin{xy}
(0,0)*+{
	\scalebox{.7}{$
	\begin{tikzpicture}
	\draw[black, ultra thick] (-1.2,0) to (-.6,-.6);
	\draw[black, ultra thick, ->] (0,-1.2) to (-.6,-.6) ;
	\draw[black, ultra thick] (0,1.2) to (.6,.6);
	\draw[black, ultra thick, ->] (1.2,0) to (.6,.6);
	\draw[black, ultra thick, ->] (1.2,0) to (.6,-.6);
	\draw[black, ultra thick] (.6,-.6) to (0,-1.2);
	\fill[black] (0,1.2) circle (.1);
	\fill[black] (0,-1.2) circle (.1);
	\fill[black] (1.2,0) circle (.1);
	\fill[black] (-1.2,0) circle (.1);
	\end{tikzpicture}
	$}
};
(11,0)*+{\mbox{{\smaller $q$}}};
(0,-15)*+{(c)};
\end{xy}
\ \ \ \ \ \ \ \ \ \ 
\!\!\!\!\!\!\!\!\!\!\!\!\!\!\!\!\!\!\!
$$
\captionsetup{singlelinecheck=off}
\caption[]{
\begin{itemize} 
\item[(a)] A metric graph $\Gamma$.
\item[(b)] A weighted graph model $G$ of $\Gamma$. 
\item[(c)] A rooted spanning tree $(T, q)$ of $G$ and the orientation $\Tc_q$.
\end{itemize}}
\label{fig:graphs}
\end{figure}

\subsection{Tropical Jacobians}\label{sec:trop_Jac}

Let $\Gamma$ be a metric graph. Fix a model $G$ of $\Gamma$. Let $\Oc=\{e_1,\ldots,e_m\}$ be a labeling of an orientation $\Oc$ on $G$. The real vector space $ C_1(G,\RR) \simeq \bigoplus_{i=1}^m  \RR e_i$ has a canonical inner product  defined by $[ e_i,e_j ]=\delta_i(j)\ell(e_i)$. Here $\delta_i$ denotes the delta (Dirac) measure on $\{ 1, 2, \ldots, m\}$ centered at $i$. The resulting inner product space $\left( C_1(G,\RR), [\cdot , \cdot] \right)$ is independent of the choice of $\Oc$ and its labeling. 

The inner product $[ \cdot,\cdot  ]$ restricts to an inner product, also denoted by $[ \cdot,\cdot  ]$, on the homology lattice $\Lambda  =H_1(G,\ZZ) \subset C_1(G,\ZZ)$. The pair $(\Lambda,[ \cdot,\cdot  ])$ is a canonical lattice associated to $\Gamma$ (independent of the choice of the model $G$), and we have a canonical identification $\Lambda \simeq H_1(\Gamma,\ZZ)$. Note that $\Lambda_{\RR} \simeq H_1(\Gamma,\RR)$. The associated polarized real torus $H_1(\Gamma,\RR)/H_1(\Gamma,\ZZ)$  is called the {\em tropical Jacobian} of $\Gamma$ (\cite{ks, mz}), and denoted by $\Jac(\Gamma)$. 

\section{Potential theory on metric graphs} \label{sec:pot}
In this section, we closely follow \cite{FR2} and review those results that are needed in this paper.

\subsection{Graphs as electrical networks}
Let $\Gamma$ be a metric graph and $G$ be a model of $\Gamma$. We may think of $\Gamma$ (or $G$) as an electrical network in which each edge $e \in E(G)$ is a resistor having resistance $\ell(e)$. See Figure~\ref{fig:electrical}.

\begin{figure}[h!]
$$
\begin{xy}
(0,0)*+{
	\scalebox{.7}{$
	\begin{tikzpicture}
	\draw[black] (-1.2,0) to[R=$2$] (0,1.2);
	\draw[black] (-1.2,0) to[R=$1$] (1.2,0);
	\draw[black] (0,-1.2) to[R=$2$] (-1.2,0);
	\draw[black] (0,1.2) to[R=$2$] (1.2,0);
	\draw[black] (1.2,0) to[R=$2$] (0,-1.2);
	\fill[black] (0,1.2) circle (.1);
	\fill[black] (0,-1.2) circle (.1);
	\fill[black] (1.2,0) circle (.1);
	\fill[black] (-1.2,0) circle (.1);
	\end{tikzpicture}
	$}
};
(0,-15)*+{N};
\end{xy}
\ \ \ \ \ \ \ \ \ \ 
\!\!\!\!\!\!\!\!\!\!\!\!\!\!\!\!\!\!\!
$$
\caption{The electrical network $N$ corresponding to the graphs in Figure~\ref{fig:graphs}~(a), (b).}\label{fig:electrical}
\end{figure}
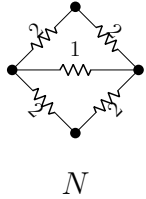

When studying the `potential theory' on a metric graph $\Gamma$, it is convenient to always fix an (arbitrary) model $G$, and think of it as an electrical network.

\subsection{Laplacians and $j$-functions} 
\label{sec:laplace}
Let $\Gamma$ be a metric graph, and let $G$ be a model of $\Gamma$.  
We have the distributional Laplacian operator (see \cite[\S3.1]{FR2})
\[\Delta \colon \PL(\Gamma) \rightarrow \dmeasz(\Gamma)\, ,\]
where $\PL(\Gamma)$ is the real vector space consisting of all continuous piecewise affine real valued functions on $\Gamma$ that can change slope finitely many times on each closed edge segment, and $\dmeasz(\Gamma)$ is the real vector space of discrete measures $\nu$ on $\Gamma$ with $\nu(\Gamma) = 0$. 
We also have the combinatorial Laplacian operator (see \cite[\S3.2]{FR2})
\[\Delta \colon \Mc(G) \rightarrow \dmeasz(G)\, ,\]
where $\Mc(G)$ is the real vector space of real-valued functions on $V(G)$, and $\dmeasz(G)$ is the real vector space of discrete measures $\nu$ on $V(G)$ with $\nu(V(G)) = 0$.
The distributional Laplacian $\Delta$ and the combinatorial Laplacian $\Delta$ are compatible in the sense described in \cite[\S3.3]{FR2}. 
Moreover, the combinatorial Laplacian on $G$ can be conveniently presented by its {\em Laplacian matrix}; let $\{v_1, \ldots , v_n\}$ be a labeling of $V(G)$. The Laplacian matrix $\Q$ associated to $G$ is the $n \times n$ matrix $\Q = (q_{ij})$ where, for $i \ne j$, we have 
$
q_{ij} = - \sum_{e =\{v_i , v_j\} \in E(G)}{{1}/{\ell(e)}} 
$.
The diagonal entries are determined by forcing the matrix to have zero-sum rows. 

The Laplacian matrix of $G$ can also be expressed in terms of the {\em incidence matrix} of $G$. Let $\{v_1, \ldots , v_n\}$ be a labeling of $V(G)$ as before. Fix an orientation $\Oc = \{e_1, \ldots, e_m\}$ on $G$. The incidence matrix $\B$ associated to $G$ is the $n \times m$ matrix $\B=(b_{ij})$, where 
$b_{ij} = +1$ if $e_j^{+} = v_i$ and $b_{ij} = -1$ if $e_j^{-} = v_i$ and $b_{ij} = 0$ otherwise. Let $\D$ denote the $m \times m$ diagonal matrix with diagonal entries $\ell(e_i)$ for $e_i \in \Oc$. We have $\Q = \B \D^{-1}\B^{\TP}$, 
where $(\cdot)^{\TP}$ denotes the matrix transpose operation.

A fundamental solution of the Laplacian is given by $j$-functions. We follow the notation of \cite{cr}. See \cite[\S4]{FR2} and references therein for more details.
Let $\Gamma$ be a metric graph and fix two points $y,z \in \Gamma$. We denote by $j_z(\cdot \, , y ; \Gamma)$ the unique function in $\PL(\Gamma)$ satisfying:
(i) $\Delta \left(j_z(\cdot \, , y; \Gamma)\right )= \delta_y - \delta_z$, and 
(ii) $j_z(z,y; \Gamma) = 0$.
If the metric graph $\Gamma$ is clear from the context we write $j_z(x,y)$ instead of $j_z(x,y; \Gamma)$.
The $j$-function exists and is unique, and satisfies the following basic properties:
\begin{itemize}
\item[$\diamond$] $j_z(x,y)$ is jointly continuous in all three variables $x,y,z \in \Gamma$.
\item[$\diamond$] $j_z(x,y) = j_z(y,x)$.
\item[$\diamond$] $0 \leq j_z(x,y) \leq j_z(x,x)$.
\item[$\diamond$] $j_z(x,x) = j_x(z,z)$.
\end{itemize}

The {\em effective resistance} between two points $x, y \in \Gamma$ is
$r(x,y) \coloneqq j_y(x,x)$.
If we want to clarify the underlying metric graph $\Gamma$, we use the notation $r(x,y; \Gamma)$.

Let $G$ be an arbitrary model of $\Gamma$. One can explicitly compute the quantities $j_q(p,v) \in \RR$ for $q,p,v \in V(G)$ using linear algebra (see \cite[\S 3]{fm}) as follows. Fix a labeling of $V(G)$ as before, and let $\Q$ be the corresponding Laplacian matrix. 
Let $\Q_q$ be the $(n-1)\times(n-1)$ matrix obtained from $\Q$ by deleting the row and column corresponding to $q \in V(G)$ from $\Q$. It is well-known that $\Q_q$ is invertible. Let $\L_q$ be the $n \times n$ matrix obtained from $\Q_q^{-1}$ by inserting zeros in the row and column corresponding to $q$. One can easily check that 
$ \Q\L_q = \I + \R_q$, 
where $\I$ is the $n \times n$ identity matrix and $\R_q$ has all $-1$ entries in the row corresponding to $q$ and has zeros elsewhere. 
It follows from the compatibility of the two Laplacians that
$\L_q = (j_q(p,v))_{p,v \in V(G)}$. The matrix $\L_q$ is a {\em generalized inverse} of $\Q$, in the sense that $\Q\L_q\Q = \Q$. 

\begin{Remark} \label{rmk:computej}
Computing $\L_q$ takes time at most $O(n^\omega)$, where $\omega$ is the exponent for matrix multiplication (currently $\omega < 2.38$).
\end{Remark}

\subsection{Cross ratios}
Let $\Gamma$ be a metric graph and fix $q \in \Gamma$. As in \cite{FR2}, we define the {\em cross ratio} function (with respect to the base point $q$) $\xi_q \colon \Gamma^4 \rightarrow \RR$  by 
\[
\xi_q(x,y , z,w) \coloneqq j_q(x,z) +j_q(y,w) - j_q(x,w) - j_q(y,z) \,.
\]
If we want to clarify the graph $\Gamma$, we use the notation $\xi_q(x,y , z,w ; \Gamma)$ instead. Cross ratios satisfy the following basic properties:
\begin{itemize}
\item[$\diamond$] $\xi(x,y , z,w)  \coloneqq \xi_q(x,y , z,w)$ is independent of the choice of $q$.
\item[$\diamond$] $\xi(x,y , z,w) = \xi( z,w,x,y)$.
\item[$\diamond$] $\xi(y,x , z,w) = - \xi( x, y, z,w)$.
\item[$\diamond$] $\xi(x,y , z,w) = \langle \delta_x-\delta_y, \delta_z-\delta_w \rangle_{\en}$, where $\langle \cdot , \cdot \rangle_{\en}$ denotes the {\em energy pairing} on $\dmeasz(\Gamma)$ defined by
\begin{equation}\label{eq:en}
\langle \nu_1 , \nu_2 \rangle_{\en} \coloneqq \int_{\Gamma \times \Gamma} j_q (x,y) \d \nu_1(x) \d \nu_2(y) \, .
\end{equation}
\end{itemize}

\begin{Example}
The following identities will be useful for our computations.

It follows from $\xi_x(q,x,q,y) = \xi_y(q,x,q,y)$ that 
\begin{equation} \label{eq:jid1}
r(x, q) - r(y, q) = j_{x}(y , q) - j_{y}(x , q)  \, .
\end{equation}

 It follows from $\xi_x(x,y,x,q) = \xi_q(x,y,x,q)$ that 
\begin{equation}  \label{eq:jid2}
j_x(y,q) =  j_q(x,x) - j_q(x,y) \, .
\end{equation}

 It follows from $\xi_x(x,y,x,q) = \xi_y(x,y,x,q)$ that 
\begin{equation}  \label{eq:jid3}
r(x,y) =  j_{x}(y , q)+ j_y(x,q) \, .
\end{equation}

 It follows from $\xi_x(x,y,x,y) = \xi_q(x,y,x,y)$ that 
\begin{equation}  \label{eq:jid4}
r(x,y) =  j_{q}(x , x)+ j_q(y,y) - 2j_q(x,y) \, .
\end{equation}

\end{Example}

\subsection{Projections} \label{sec:proj}
Let $G$ be a weighted graph. Fix an orientation $\Oc$. Let $T$ be a spanning tree of $G$. 
The {\em weight} of $T$ is the product $
w(T) \coloneqq \prod_{e \not \in T}{\ell(e)}$. 
The {\em coweight} of $T$ is the product $
w'(T) \coloneqq  \prod_{e \in T}{\ell^{-1}(e)}$.
The {\em weight} and {\em coweight} of $G$ are
$w(G) \coloneqq \sum_{T}{w(T)}$ and $w'(G) \coloneqq \sum_{T}{w'(T)}$,
where the sums are over all spanning trees of $G$. 
The quantity $w(G)$ depends only on the underlying metric graph $\Gamma$. 

Let $\M_T$ be the $m \times m$ matrix whose columns are obtained from $1$-chains $\circu(T,e)$ associated to fundamental circuits of $T$, and let $\N_T$ be the $m \times m$ matrix whose columns are obtained from $1$-chains $\cocirc(T,e)$ associated to fundamental cocircuits of $T$ (see \cite[\S7.2.1]{FR2}).
Consider the following matrix averages:
\[
\P = \sum_{T}{\frac{w(T)}{w(G)} \M_T} \quad , \quad \P' = \sum_{T}{\frac{w'(T)}{w'(G)} \N_T} \, ,
\]
the sums being over all spanning trees $T$ of $G$. 
It is a classical theorem of Kirchhoff \cite{Kirchhoff} that the matrix of $\pi \colon C_1(G,\RR) \twoheadrightarrow H_1(G, \RR)$, with respect to $\Oc$, is $\P$.
Similarly, the matrix of $\pi' \colon C_1(G,\RR) \twoheadrightarrow H_1(G, \RR)^{\perp}$, with respect to $\Oc$, is $(\P')^{\TP}$. 

Let $\Xib$ be the $m \times m$ {\em matrix of cross ratios}:
\[
\Xib \coloneqq \left(\xi(e^-, e^+ , f^- , f^+) \right)_{e,f \in \Oc} \, . 
\]
Let $\L$ be any generalized inverse of $\Q$ (i.e. $\Q\L\Q = \Q$). Then we have $\Xib = \B^{\TP} \L \B$.
It is shown in \cite[Proposition~7.8]{FR2} that the matrix of $\pi \colon C_1(G,\RR) \twoheadrightarrow H_1(G, \RR)$, with respect to $\Oc$, is 
$\I- \D^{-1} \Xib$, and the matrix of $\pi' \colon C_1(G,\RR) \twoheadrightarrow H_1(G, \RR)^{\perp}$, with respect to $\Oc$, is 
$\D^{-1} \Xib$. In particular, for each $f \in \Oc$ we have
\begin{itemize}
\item[$\diamond$] $\pi(f) = \sum_{e \in \Oc} {\F(e,f) e}$, where 
\begin{equation} \label{eq:Fef}
\F(e,f) \coloneqq
\begin{cases}
 1-{r(e^-, e^+)}/{\ell(e)}  &\text{ if } e=f\\
-{\xi(e^-, e^+ ,f^-, f^+)}/{\ell(e)}  &\text{ if } e \ne f \, .
\end{cases} 
\end{equation}
\item[$\diamond$] $\pi'(f) = \sum_{e \in \Oc} {\F'(e,f) e}$, where 
\[
\F'(e,f) =
{\xi(e^-, e^+ ,f^-, f^+)}/{\ell(e)} \, .
\]
\end{itemize}
Moreover, we have equalities:
\begin{equation} \label{eq:Ps_Cross}
\P= \I- \D^{-1} \Xib \quad , \quad \P' = \Xib \D^{-1}\, .
\end{equation}

\begin{Definition} \label{def:foster}
The {\em Foster coefficient} of $e \in \mathbb{E}(G)$ is, by definition,
\[\F(e) \coloneqq \F(e,e) = 1-\frac{r(e^-, e^+)}{\ell(e)}\, . \]
Clearly, $\F(e) = \F(\bar{e})$, so $\F(e)$ is also well-defined for $e \in E(G)$. 
\end{Definition}
\begin{Remark} \phantomsection \label{rmk:foster}
\begin{itemize} 
\item[]
\item[(i)] It follows from \eqref{eq:Ps_Cross} that 
\[
\F(e) = \sum_{T \not \ni e} \frac{w(T)}{w(G)}\, ,
\]
the sum being over all spanning trees $T$ of $G$ {\em not containing} $e$.
\item[(ii)] It is a consequence of `Rayleigh's monotonicity law' that $0 \leq \F(e) < 1$, and the equality $\F(e)=0$ holds if and only if $e$ is a bridge.
\end{itemize}
\end{Remark}
Fix an arbitrary path $\gamma$ from $y$ to $x$. Let $\boldsymbol{\gamma}_{yx}$ denote the associated $1$-chain. Then, by \cite[Corollary~7.13]{FR2}, we have
\begin{equation} \label{eq:resist}
r(x,y) = [\boldsymbol{\gamma}_{yx}, \pi'(\boldsymbol{\gamma}_{yx})] \, .
\end{equation}
\begin{Example} \label{ex:eff}
The following observation will be useful for computations. Let $e = \{u,v\}$ denote an edge segment in a metric graph $\Gamma$, and let $p \in e$ be a point with distance $x$ from $u$ and distance $\ell(e) - x$ from $v$. Then, for each point $q \in \Gamma$, we have
\[
r(p,q) = \frac{\ell(e)-x}{\ell(e)} r(u,q) + \frac{x}{\ell(e)}r(v,q) + \F(e) \frac{\left(\ell(e)-x\right)x}{\ell(e)}\, .
\]
This follows, for example, by a direct computation using \eqref{eq:resist}. We leave the details to the interested reader.
\end{Example}
\subsection{Generalized Rayleigh's laws}
We will need the following two results from \cite[\S8]{FR2}. Let $\Gamma$ be a metric graph. Let $e$ be an edge segment of $\Gamma$ with boundary points $\partial e = \{e^-, e^+\}$. Let $\Gamma/e$ denote the metric graph obtained by contracting $e$ (equivalently, by setting $\ell(e) = 0$). Then
\begin{equation}\label{eq:Rayleigh_j}
j_z(x,y ; \Gamma/e) = j_z(x,y ; \Gamma) -  \frac{\xi(x,z , e^-, e^+ ; \Gamma) \, \xi(y,z , e^-, e^+; \Gamma)}{r(e^-, e^+; \Gamma)} \, ,
\end{equation}
and
\begin{equation}\label{eq:Rayleigh_r}
r(x,y ; \Gamma/e) =r(x,y ; \Gamma) -  \frac{\xi(x,y , e^-, e^+; \Gamma) ^2}{r(e^-, e^+; \Gamma)} \, .
\end{equation}
Note that \eqref{eq:Rayleigh_r} is, in fact, a special case of \eqref{eq:Rayleigh_j}.

\subsection{Contractions and models}\label{sec:contr_and_models}
Let $G$ be a weighted graph and let $e \in E(G)$ be an edge of $G$.
We denote by $G/e$ the weighted graph obtained from $G$ by contracting the edge $e$ and removing all loops that might be created in the process. Assume that $G$ is a model of the metric graph $\Gamma$. In particular we may view $e$ as an edge segment of $\Gamma$. We then observe that for $x, y, z, w \in V(G)$ the cross ratio $\xi(x,y,z,w;\Gamma/e)$ measured on $\Gamma/e$ is equal to the cross ratio $\xi(x,y,z,w;G/e)$ measured on (the metric graph canonically associated to) $G/e$. A similar remark pertains to the $j$-function $j_z(x,y;\Gamma/e)$ and the effective resistance function $r(x,y;\Gamma/e)$. We leave the details to the reader. 

\section{Calculus of random spanning trees} \label{sec:random}

In this section we start with the real work leading to Theorem~\ref{thm:momentIntro}.
\begin{Definition}
Let $G$ be a model of a metric graph $\Gamma$. For edges $e = \{e^-, e^+\}$ and $f = \{f^-, f^+\}$ of $G$ we define: 
\[
\Prm (e,f) \coloneqq 
\begin{cases}
{r(e^-, e^+ ; G)}/{\ell(e)} &\text{ if $e=f$,}\\
{r(e^-, e^+ ; G)}/{\ell(e)} \times {r(f^-, f^+ ; G/e)}/{\ell(f)} &\text{ if $e\ne f$.}
\end{cases}
\]
We use the notation $\Prm(e)\coloneqq \Prm (e,e)$. If we want to clarify the underlying model $G$, we use the notations $\Prm(e,f; G)$ and $\Prm(e; G)$.
\end{Definition}
By Definition~\ref{def:foster} and Remark~\ref{rmk:foster}~(i) we know
\begin{equation} \label{eq:Pe_sum}
\Prm(e) = 1-\F(e) = \sum_{T \ni e} \frac{w(T)}{w(G)}\, ,
\end{equation}
the sums being over all spanning trees $T$ of $G$ {\em containing} $e$. So $\Prm(e)$ is the probability of $e$ being present in a {\em random spanning tree}, where a spanning tree $T$ is chosen with probability ${w(T)}/{w(G)}$. 

A similar probabilistic interpretation holds for $\Prm (e,f)$ when $e\ne f$. Namely, since $\Prm (e,f) = \Prm(e; G) \Prm(f; G/e)$, it represents the probability  
of both $e$ and $f$ being present in a random spanning tree. In other words, 
\begin{equation} \label{eq:Pef_sum}
\Prm(e,f) = \sum_{T \ni e,f} \frac{w(T)}{w(G)}\, .
\end{equation}
It follows that $\Prm(e,f) = \Prm(f,e)$. One can use \eqref{eq:Rayleigh_r} (alternatively, the `transfer--current theorem' -- see \cite[\S4.2]{LP}) to compute $\Prm(e,f)$ directly in terms of invariants of $G$.

\begin{Definition}\phantomsection \label{def:stars}
\begin{itemize}
\item[]
\item[(i)] Let $\st \colon V(G) \rightarrow \RR$ be the function defined by sending $p \in V(G)$ to
\[
\st(p)  \coloneqq \sum_{e=\{p , x\}} \Prm(e) \, ,
\]
the sum being over all edges $e$ incident to $p$ in $G$ (i.e. the {\em star} of $p$).
\item[(ii)] Let $\tf \colon V(G) \times E(G) \rightarrow \RR$ be the function defined by sending $(p , e)$ to
\[
\tf(p, e) \coloneqq \sum_{f=\{p , x\}} \Prm(e, f) \, ,
\]
the sum being over all edges $f$ incident to $p$ in $G$.
\end{itemize}
\end{Definition}

\begin{Proposition} \label{prop:starformula}
Fix a vertex $q \in V(G)$. We have
\[
\st(p) = \sum_{e=\{p , x\}} \frac{r(x,q) - r(p,q)}{\ell(e)} + 2 - 2\, \delta_q(p) \, .
\]
the sum being over all edges $e \in E(G)$ incident to $p$. 
\end{Proposition}
\begin{proof}
By \eqref{eq:jid4} we may write $r(p,x) = j_q(x,x) - j_q(p,p) + 2 \left(j_q(p,p) - j_q(x,p)\right)$. Therefore
\[
\begin{aligned}
\st(p) &= \sum_{e=\{p , x\}} \frac{r(p,x)}{\ell(e)} \\
&= \sum_{e=\{p , x\}} \frac{ j_q(x,x) - j_q(p,p)}{\ell(e)} + 2 \sum_{e=\{p , x\}}\frac{j_q(p,p) - j_q(x,p)}{\ell(e)} \\
&= \sum_{e=\{p , x\}} \frac{ j_q(x,x) - j_q(p,p)}{\ell(e)} + 2 \Delta \left( j_q(\cdot,p)\right) (p)\\
&= \sum_{e=\{p , x\}} \frac{ j_q(x,x) - j_q(p,p)}{\ell(e)} + 2 (\delta_p(p) - \delta_q(p)) \, .
\end{aligned}
\]
\end{proof}
Recall from \S\ref{sec:contr_and_models}  the weighted graph $G/e$ obtained by contracting the edge $e=\{u,v\} \in E(G)$ and removing all loops that might be created in the process. Let $\eb \subseteq E(G)$ denote the set of all edges {\em parallel} to $e$ (i.e. connecting $u$ and $v$). We make the identification $E(G/e) = E(G)\backslash \eb$, and the two vertices $u, v \in V(G)$ will be identified with a single vertex $v_e \in V(G/e)$ (see Figure~\ref{fig:contract}), and $V(G) \backslash \{u,v\} = V(G/e) \backslash \{v_e\}$.

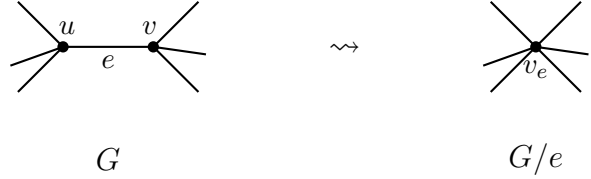
\begin{figure}[h!]
$$
\begin{xy}
(0,0)*+{
	\scalebox{.5}{$
	\begin{tikzpicture}
	\draw[black, ultra thick] (-1.2,0) to (-2.4,1.2);
	\draw[black, ultra thick] (-1.2,0) to (1.2,0);
	\draw[black, ultra thick] (-1.2,0) to (-2.4,-1.2);
	\draw[black, ultra thick] (-1.2,0) to (-2.6,-.5);
	\draw[black, ultra thick] (1.2,0) to (2.4,1.2);
	\draw[black, ultra thick] (1.2,0) to (2.4,-1.2);
	\draw[black, ultra thick] (1.2,0) to (2.6,-.2);
	\fill[black] (1.2,0) circle (.15);
	\fill[black] (-1.2,0) circle (.15);
	\end{tikzpicture}
	$}
};
(-5.5,2.5)*+{\mbox{{\smaller $u$}}};
(5.5,2.5)*+{\mbox{{\smaller $v$}}};
(0,-2)*+{\mbox{{\smaller $e$}}};
(0,-15)*+{G};
\end{xy}
\ \ \ \ \ \ \ \ \ \ \rightsquigarrow \ \ \ \ \ \ \ \ \ \ 
\begin{xy}
(0,0)*+{
	\scalebox{.5}{$
	\begin{tikzpicture}
	\draw[black, ultra thick] (0,0) to (-1.2,1.2);
	\draw[black, ultra thick] (0,0) to (-1.2,-1.2);
	\draw[black, ultra thick] (0,0) to (-1.4,-.5);
	\draw[black, ultra thick] (0,0) to (1.2,1.2);
	\draw[black, ultra thick] (0,0) to (1.2,-1.2);
	\draw[black, ultra thick] (0,0) to (1.4,-.2);
	\fill[black] (0,0) circle (.15);
	\end{tikzpicture}
	$}
};
(0,-3)*+{\mbox{{\smaller $v_e$}}};
(0,-15)*+{G/e};
\end{xy}
\ \ \ \ \ \ \ \ \ \ 
\!\!\!\!\!\!\!\!\!\!\!\!\!\!\!\!\!\!\!
$$
\caption{Contracting the edge $e = \{u,v\}$.}
\label{fig:contract}
\end{figure}

\begin{Theorem} \label{thm:star_ve_formula}
Fix a vertex $q \in V(G)$. Let  $p \in V(G)$ and $e= \{u,v\} \in E(G)$. Then
\[
\begin{aligned}
\tf(p,e) &= \frac{r(u,v; G)}{\ell(e)} \sum_{\substack{f=\{p , x\}  \\ f \in E(G)\backslash \eb}} \frac{r(x,q; G/e) - r(p,q; G/e)}{\ell(f)}   \\
 &+2 \frac{r(u,v; G)}{\ell(e)} \left(1- \delta_q(p) - \frac{\xi(u,v,u,q; G)}{r(u,v; G)} \delta_u(p)- \frac{\xi(v,u,v,q; G)}{r(u,v; G)}\delta_v(p) \right)
\, .
\end{aligned}
\]
The sum is over all edges $f \in E(G) \backslash \eb$ incident to $p$ in $G$. 

\end{Theorem}
\begin{proof}
By definition 
\[
\Prm (e,f) = \frac{r(u,v; G)}{\ell(e)} \Prm(f; G/e)\, .\] 
If $p \not \in \{u,v\}$ the result follows immediately from Proposition~\ref{prop:starformula}:
\[
\tf(p,e) = \frac{r(u,v; G)}{\ell(e)} \left(\sum_{\substack{f=\{p , x\}  \\ f \in E(G)}} \frac{r(x,q; G/e) - r(p,q; G/e)}{\ell(e)} + 2 - 2\, \delta_q(p)\right) \, ,
\]
the sum being over all edges $f \in E(G)$ incident to $p$ in $G$ (equivalently, all edges $f \in E(G/e)$ incident to $p$ in $G/e$). 

So, by symmetry, it remains to show the equality for $p= u$. Recall we denote the vertex obtained by identifying $u$ and $v$ in $G/e$ by $v_e$ (Figure~\ref{fig:contract}). As in the proof of Proposition~\ref{prop:starformula}, we first compute:
\begin{equation} \label{eq:pu1}
\begin{aligned}
\tf(u,e) &= \frac{r(u,v; G)}{\ell(e)} \left(\sum_{\substack{f=\{u , x\} \\ f \in E(G)\backslash \eb}} \frac{ j_q(x,x;G/e) - j_q(v_e,v_e;G/e)}{\ell(f)} \right) \\
&+ 2 \frac{r(u,v; G)}{\ell(e)}\left(\sum_{\substack{f=\{u , x\} \\ f \in E(G)\backslash \eb}}\frac{j_q(v_e,v_e;G/e) - j_q(x,v_e;G/e)}{\ell(f)} \right) \, ,
\end{aligned}
\end{equation}
where the sums are over all edges $f \in E(G) \backslash \eb$ incident to $u$ in $G$. Unlike in the proof of Proposition~\ref{prop:starformula}, we {\em cannot} interpret the second sum as the Laplacian of the $j$-function on $G/e$ because the summation is not over all edges incident to $v_e$ in $G/e$ (e.g. in Figure~\ref{fig:contract} only edges on the left of $v_e$ appear in the summation). We proceed by `lifting' the problem to $G$ using generalized Rayleigh's laws \eqref{eq:Rayleigh_j} and \eqref{eq:Rayleigh_r}. We find
\[
\begin{aligned}
&\frac{j_q(v_e,v_e;G/e) - j_q(x,v_e;G/e)}{\ell(f)} = \frac{j_q(u,u;G) - j_q(x,u;G)}{\ell(f)}\\
&+ \frac{\xi(u,v,u,q;G)}{r(u,v;G)}\left( \frac{\xi(u,v,x,q;G) - \xi(u,v,u,q;G)}{\ell(f)}\right) \, .
\end{aligned}
\]
It is easily checked that the right hand side is zero for $x=v$. Moreover, by the definition of cross ratios, we compute
\[\xi(u,v,x,q;G) - \xi(u,v,u,q;G) = \xi(u,v,x,u;G)\, .\] 
So we have
\begin{equation} \label{eq:pu2}
\begin{aligned}
&\sum_{\substack{f=\{u , x\} \\ f \in E(G)\backslash \eb}}\frac{j_q(v_e,v_e;G/e) - j_q(x,v_e;G/e)}{\ell(f)} = \sum_{\substack{f=\{u , x\} \\ f \in E(G)}} \frac{j_q(u,u;G) - j_q(x,u;G)}{\ell(f)}\\
&+ \frac{\xi(u,v,u,q;G)}{r(u,v;G)} \sum_{\substack{f=\{u , x\} \\ f \in E(G)}} \frac{\xi(u,v,x,u;G)}{\ell(f)} \, .
\end{aligned}
\end{equation}
As in the proof of Proposition~\ref{prop:starformula} we have 
\begin{equation} \label{eq:pu3}
\begin{aligned}
\sum_{\substack{f=\{u , x\} \\ f \in E(G)}} \frac{j_q(u,u;G) - j_q(x,u;G)}{\ell(f)} &= \Delta\left(j_q (\cdot , u)\right)(u)\\
&= \delta_u(u)-\delta_q(u) = 1- \delta_q(u)\, .
\end{aligned}
\end{equation}
By expanding with respect to the base point $q$, we also compute 
\begin{equation} \label{eq:pu4}
\begin{aligned}
\sum_{\substack{f=\{u , x\} \\ f \in E(G)}} \frac{\xi(u,v,x,u;G)}{\ell(f)} &= 
\sum_{\substack{f=\{u , x\} \\ f \in E(G)}} \frac{j_q(v,u;G) - j_q(v,x;G)}{\ell(f)} 
\\
&- \sum_{\substack{f=\{u , x\} \\ f \in E(G)}} \frac{j_q(u,u;G) - j_q(u,x;G)}{\ell(f)}
\\
&= \Delta\left(j_q (v , \cdot)\right)(u) - \Delta\left(j_q (u , \cdot)\right)(u)
\\
&= \left(\delta_v(u) - \delta_q(u) \right) -  \left(\delta_u(u) - \delta_q(u) \right) \\
&= -1
\, .
\end{aligned}
\end{equation}
The result for $p=u$ follows by putting together \eqref{eq:pu1}, \eqref{eq:pu2}, \eqref{eq:pu3}, \eqref{eq:pu4}.
\end{proof}

\section{Energy levels of rooted spanning trees}
\label{sec:crossavg}
In this section, we prove a very subtle identity for cross ratios (Theorem \ref{thm:energylevel}). 

\subsection{Energy levels}
Recall that a rooted spanning tree $(T, q)$ of $G$ comes with a preferred orientation $\Tc_q \subseteq \mathbb{E}(G)$, where all edges are oriented {\em away from $q$} on the spanning tree $T$ (see \S\ref{sec:weightedgraph}).
\begin{Definition} \label{def:enlev}
We define the {\em energy level} of a rooted spanning tree $(T, q)$ to be 
\[
\boldsymbol{\varrho}(T, q) \coloneqq \sum_{e,f \in \Tc_q} \xi(e^-, e^+, f^-, f^+) \, .
\]
\end{Definition}
The following result is a justification for our terminology, and is useful in our later computation. Let $\deg_T(v)$ denote the number of edges incident with $v \in V(G)$ in the spanning tree $T$. Consider the canonical element $ \nu_{(T,q)} \in \dmeasz(G)$ associated to the rooted spanning tree $(T, q)$ of $G$ defined by
\[
\nu_{(T,q)} = \sum_{v\in V(G)} \left(\deg_T(v) - 2\right) \delta_v + 2 \delta_q \, .
\] 
\begin{Lemma} \label{lem:rootedenergy}
We have
 \begin{equation} \label{eq:rhovsen}
 \boldsymbol{\varrho}(T, q ) = \langle \nu_{(T,q)}, \nu_{(T,q)} \rangle_{\en}
= \sum_{v,w \in V(G)} \left(\deg_T(v) - 2\right)\left(\deg_T(w) - 2\right) j_q(v,w)\, .
\end{equation}
\end{Lemma}
\begin{proof}
By the definition of the preferred orientation $\Tc_q \subseteq \mathbb{E}(G)$  we have:
\[
\sum_{e \in \Tc_q} \left(\delta_{e^-} - \delta_{e^+}\right) = \nu_{(T,q)}\,.
\]
This is because every vertex $v \ne q$ has exactly $1$ incoming edge and $\deg_T(v) - 1$ outgoing edges in $\Tc_q$. At $q$, however, there is no incoming edge in $\Tc_q$.

The result now follows directly from bilinearity of the energy pairing:
\[
\begin{aligned}
\boldsymbol{\varrho}(T, q) &= \sum_{e,f \in \Tc_q} \xi(e^-, e^+, f^-, f^+)
=\sum_{e,f \in \Tc_q} \langle \delta_{e^-} - \delta_{e^+}, \delta_{f^-} - \delta_{f^+}\rangle_{\en} \\
&=\ \langle \, \sum_{e \in \Tc_q} \left(\delta_{e^-} - \delta_{e^+}\right),\sum_{f \in \Tc_q}\left( \delta_{f^-} - \delta_{f^+}\right) \, \rangle_{\en} 
=  \langle \nu_{(T,q)}, \nu_{(T,q)} \rangle_{\en}
\, .
\end{aligned}
\]
The second equality in \eqref{eq:rhovsen} follows from \eqref{eq:en}, because $j_q(\cdot , q) = j_q (q, \cdot) = 0$.
\end{proof}

\subsection{The average of energy levels}
We can now state the most technical ingredient in computing tropical moments of tropical Jacobians.

\begin{Theorem}\label{thm:energylevel}
Fix a vertex $q \in V(G)$. We have the equality
\[
\frac{1}{w(G)}\sum_T {w(T) \, {\boldsymbol{\varrho}}(T, q)} = \sum_{e = \{u,v\}}\frac{j_{u}(v, q)^2 + j_{v}(u, q)^2}{\ell(e)}\, ,
\]
where the sum on the left is over all spanning trees $T$ of $G$, and the sum on the right is over all edges $e \in E(G)$.
\end{Theorem}
\begin{proof}
The result follows by putting together Lemma~\ref{lem:l2}, Lemma~\ref{lem:l3}, and Lemma \ref{lem:l4} below.  
\end{proof}

\begin{Lemma} \label{lem:l1} 
Let $(T,q)$ be a rooted spanning tree of $G$. We have the equality
\[
\begin{aligned}
\boldsymbol{\varrho}(T, q) &= 4 \sum_{x,y \in V(G)} j_q(x,y)
- 4 \sum_{e \in T} \sum_{x \in V(G)} \left(j_q(e^-, x)+j_q(e^+, x)\right) \\
&+ \sum_{e,f \in T}\left(j_q(e^-, f^-)+j_q(e^+, f^+)+j_q(e^-, f^+)+j_q(e^+, f^-) \right)\, .
\end{aligned}
\]
\end{Lemma}
\begin{proof}
The result follows from Lemma~\ref{lem:rootedenergy}, and an application of a (generalized) {\em handshaking lemma}: for each $\psi \in \Mc(G)$ we have 
\[ \sum_{x \in V(G)} \deg_T(x) \,\psi (x) = \sum_{e \in T} \left(\psi(e^-)+\psi(e^+)\right)\, .\]
\end{proof}

\begin{Lemma} \label{lem:l2}
Fix a vertex $q \in V(G)$. We have the equality
\[
\begin{aligned}
\frac{1}{w(G)}&\sum_T {w(T) \, {\boldsymbol{\varrho}}(T, q)} =  4 \!\!\!\! \!\sum_{x,y \in V(G)} j_q(x,y) -4\sum_{e \in E(G)} \sum_{x \in V(G)} \left(j_q(e^-, x)+j_q(e^+, x)\right) \Prm(e) \\
&+\sum_{e,f \in E(G)}\left(j_q(e^-, f^-)+j_q(e^+, f^+)+j_q(e^-, f^+)+j_q(e^+, f^-)\right) \Prm(e,f) \, .
\end{aligned}
\]
\end{Lemma}
\begin{proof}
This follows from Lemma~\ref{lem:l1}: one only needs to change the order of summations and use \eqref{eq:Pe_sum} and \eqref{eq:Pef_sum}. 
\end{proof}

To prove the next two results it is convenient to work with matrices. We fix labelings $V(G) = \{v_1, \ldots , v_n\}$ and $E(G) = \{e_1, \ldots, e_m\}$, and we introduce the following matrices:

\begin{itemize}[noitemsep,nolistsep, leftmargin=*]
\item[$\diamond$] $\ones_n \in \RR^n$ and $\ones_m \in \RR^m$ denote the all-ones vectors.
\item[$\diamond$] $\db_v \in \RR^n$ (resp. $\db_e \in \RR^m$) denotes the characteristic vector of $v \in V(G)$ (resp. $e \in E(G)$). 
\item[$\diamond$] $\Ab = (h_{ij})$ denotes the $n \times m$ {\em unsigned} incidence matrix of $G$, where $h_{ij} = 1$ if $e_j^+ = v_i$ or $e_j^- = v_i$, and $h_{ij} = 0$ otherwise. 
\item[$\diamond$] $\Xb$ is the $n \times n$ diagonal matrix with diagonal $(i,i)$-entries $j_q(v_i, v_i) = r(v_i, q)$. 
\item[$\diamond$] $\Yb$ is the $m \times m$ diagonal matrix with diagonal $(i,i)$-entries $\Prm(e_i)$. 
\item[$\diamond$] $\Zb$ is the $m \times m$ matrix whose diagonal entries are zero, and the $(i,j)$-entries ($i\ne j$) are $\Prm(e_i, e_j)$.
\end{itemize}
We will also use the matrices $\Q$, $\L_q$, $\R_q$, and $\I$ introduced in \S\ref{sec:laplace}.
\begin{Lemma} \label{lem:l3}
We have the equality
\begin{equation} \label{eq:2ndterm}
\sum_{e \in E(G)} \sum_{x \in V(G)} \left(j_q(e^-, x)+j_q(e^+, x)\right)\Prm(e) = 2 \sum_{x,y \in V(G)} j_q(x,y) - \sum_{x \in V(G)} j_q(x,x) \, .
\end{equation}
\end{Lemma}
\begin{proof}
We start by writing the left-hand side of \eqref{eq:2ndterm} in terms of our matrices:
\[
\sum_{e \in E(G)} \sum_{x \in V(G)} \left(j_q(e^-, x)+j_q(e^+, x)\right)\Prm(e) = \ones_n^{\TP}\L_q \Ab \Yb \ones_m \, . 
\]
By definitions we observe $\Ab \Yb \ones_m = \left(\st(v_1), \ldots , \st(v_n)\right)^{\TP}$. So, by Proposition~\ref{prop:starformula}, we have
\[
\Ab \Yb \ones_m = -\Q \Xb \ones_n + 2\left(\ones_n - \db_q\right)
\]
Therefore, 
\[
\ones_n^{\TP}\L_q \Ab \Yb \ones_m = - \ones_n^{\TP}\L_q\Q \Xb \ones_n + 2 \ones_n^{\TP}\L_q \ones_n -2\ones_n^{\TP}\L_q \db_q
\]
Note that $\L_q \db_q = \mathbf{0}$ and $\L_q\Q = \I+\R_q^{\TP}$ (see \S\ref{sec:laplace}). We also have $\R_q^{\TP} \Xb = \mathbf{0}$. Therefore,
\[
\ones_n^{\TP}\L_q \Ab \Yb \ones_m = 2 \ones_n^{\TP}\L_q \ones_n - \ones_n^{\TP} \Xb \ones_n
\]
which is the right-hand side of \eqref{eq:2ndterm}.
\end{proof}

For our next computation, it is convenient to use the notion of {\em Hadamard--Schur products} of matrices: for two $k \times m$ matrices  $A , B$ the Hadamard--Schur product, denoted by $A \circ B$, is a matrix of the same dimension as $A$ and $B$ with entries given by
$\left(A \circ B\right)_{ij} = (A)_{ij}(B)_{ij}$. 

One useful (and easy to prove) fact about Hadamard--Schur products is the following:
\begin{equation} \label{SchurTrace}
\ones_k^{\TP} \left( A \circ B\right) \ones_m= \Tr \left({A^{\TP}B}\right) \, .\end{equation} 
\begin{Lemma} \label{lem:l4}
We have the equality
\[
\begin{aligned}
&\sum_{e,f \in E(G)}\left(j_q(e^-, f^-)+j_q(e^+, f^+)+j_q(e^-, f^+)+j_q(e^+, f^-)\right) \Prm(e,f) = \\
&=\sum_{e = \{u,v\}}\frac{j_{u}(v, q)^2 + j_{v}(u, q)^2}{\ell(e)} +4 \sum_{x,y \in V(G)} j_q(x,y) -4 \sum_{x \in V(G)} j_q(x,x)
\, .
\end{aligned}
\]
\end{Lemma}
\begin{proof}
We first write:
\begin{equation} \label{eq:diagoffdiag}
\begin{aligned}
&\sum_{e,f \in E(G)} \left(j_q(e^-, f^-)+j_q(e^+, f^+)+j_q(e^-, f^+)+j_q(e^+, f^-)\right) \Prm(e,f) = \\
&= \sum_{e\in E(G)} \left(j_q(e^-, e^-)+j_q(e^+, e^+)+2j_q(e^-, e^+)\right) \Prm(e)\\
&+ \sum_{\substack{e,f \in E(G)\\ e \ne f}} \left(j_q(e^-, f^-)+j_q(e^+, f^+)+j_q(e^-, f^+)+j_q(e^+, f^-)\right) \Prm(e,f)  \, .
\end{aligned}
\end{equation}
We write the second sum in \eqref{eq:diagoffdiag} in terms of our matrices:
\begin{equation} \label{eq:offdiag}
\begin{aligned}
\sum_{\substack{e,f \in E(G)\\ e \ne f}} &\left(j_q(e^-, f^-)+j_q(e^+, f^+)+j_q(e^-, f^+)+j_q(e^+, f^-)\right) \Prm(e,f) \\
&= \ones_m^{\TP}\left( \left(\Ab^{\TP} \L_q \Ab \right) \circ \Zb\right) \ones_m \, .
\end{aligned}
\end{equation}
By \eqref{SchurTrace}, we know
\[
\ones_m^{\TP}\left( \left(\Ab^{\TP} \L_q \Ab \right) \circ \Zb\right) \ones_m = \Tr \left( \Ab^{\TP} \L_q \Ab \Zb \right) = \sum_{e\in E(G)} \db_e^{\TP} \left(\Ab^{\TP} \L_q \Ab \Zb \right) \db_e\, .
\]
By definitions, one observes $\Ab \Zb \db_e = \left(\tf(v_1, e), \ldots , \tf(v_n, e)\right)^{\TP}$. Let 
\[
e=\{u,v\} \ \ , \ \  \beta_u = \frac{\xi(u,v,u,q)}{r(u,v)}  \ \ , \ \    \beta_v = \frac{\xi(v,u,v,q)}{r(u,v)}  \ \ , \ \   \Prm(e) = \frac{r(u,v)}{\ell(e)} \, .
\] 
Theorem~\ref{thm:star_ve_formula} states:
\[
\Ab \Zb \db_e = \Prm(e) \left(-\Q \widetilde{\Xb} \ones_n +2\left(\ones_n - \db_q - \beta_u \db_u - \beta_v \db_v\right) \right) \, .
\]
Here $\widetilde{\Xb}$ is the $n \times n$ diagonal matrix with diagonal $(i,i)$-entries $j_q(v_i, v_i; G/e)$ for $v_i \not \in \{u,v\}$. The diagonal entries corresponding to both $u$ and $v$ are $j_q(v_e, v_e; G/e)$. Thus we find
\[
\db_e^{\TP} \left(\Ab^{\TP} \L_q \Ab \Zb \right) \db_e = \Prm(e) \left(\db_u+\db_v\right)^{\TP} \L_q \left(-\Q \widetilde{\Xb} \ones_n  +2\left(\ones_n - \db_q - \beta_u \db_u - \beta_v \db_v\right) \right) \, .
\]
Note that $\L_q \db_q = \mathbf{0}$, $\L_q\Q = \I+\R_q^{\TP}$, and $\R_q^{\TP} \widetilde{\Xb} = \mathbf{0}$. We obtain
\begin{equation} \label{eq:edgeterm}
\begin{aligned}
&\db_e^{\TP} \left(\Ab^{\TP} \L_q \Ab \Zb \right) \db_e  =\Prm(e) \left(\db_u+\db_v\right)^{\TP} \left(-\widetilde{\Xb} \ones_n 
+ 2  \L_q \left(\ones_n -    \beta_u \db_u +\beta_v \db_v\right) \right) \\
 &= -\Prm(e) \left(j_q(v_e,v_e; G/e)+j_q(v_e,v_e; G/e)\right) +2 \Prm(e) \! \! \sum_{x \in V(G)}\left(j_q(u,x) + j_q(v, x)\right)\\
 &\ \ \ -2 \left(\frac{\xi(u,v,u,q)}{\ell(e)}\left(j_q(u,u) + j_q(v, u)\right)
+\frac{\xi(v,u,v,q)}{\ell(e)} \left(j_q(u,v) + j_q(v, v)\right)\right)
\, .
\end{aligned}
\end{equation}
We now use our generalized Rayleigh's law \eqref{eq:Rayleigh_r} twice and write everything in terms of invariants of $G$:
\begin{equation} \label{eq:ral}
\begin{aligned}
j_q(v_e,v_e; G/e) &= j_q(u,u) - \frac{\xi(u,q,u,v)^2}{r(u,v)} =  j_q(u,u) - \frac{j_u(v,q)^2}{r(u,v)}\\
j_q(v_e,v_e; G/e) &= j_q(v,v) - \frac{\xi(v,q,u,v)^2}{r(u,v)} = j_q(v,v) - \frac{j_v(u,q)^2}{r(u,v)} \,.
\end{aligned}
\end{equation}
We also use the definition of cross ratios to compute:
\begin{equation} \label{eq:crs}
\xi(u,v,u,q) = j_q(u,u)- j_q(u,v) \   ,  \  
\xi(v,u,v,q) = j_q(v,v)- j_q(u,v) \, .
\end{equation}
Putting together \eqref{eq:diagoffdiag}, \eqref{eq:offdiag}, \eqref{eq:edgeterm}, \eqref{eq:ral}, \eqref{eq:crs}, with a simple computation we obtain:
\begin{equation} \label{eq:almostdone1}
\begin{aligned}
&\sum_{e,f \in E(G)} \left(j_q(e^-, f^-)+j_q(e^+, f^+)+j_q(e^-, f^+)+j_q(e^+, f^-)\right) \Prm(e,f) =\\
&= \sum_{e = \{u,v\}}\frac{j_{u}(v, q)^2 + j_{v}(u, q)^2}{\ell(e)} + 2  \sum_{e = \{u,v\}}    \sum_{x \in V(G)} \Prm(e)\left( j_q(u,x)+j_q(v,x)\right) \\
&+ 2 \sum_{e =\{u,v\}} \Prm(e) j_q(u,v)
-2  \sum_{e = \{u,v\}} \frac{j_q(u,u)^2 +j_q(v,v)^2 -2j_q(u,v)^2}{\ell(e)} \,.
\end{aligned}
\end{equation}
By Lemma~\ref{lem:l3}, the second term in \eqref{eq:almostdone1} is simplified as 
\begin{equation}\label{eq:almostdone2}
 \sum_{e = \{u,v\}}    \sum_{x \in V(G)} \Prm(e)\left( j_q(u,x)+j_q(v,x)\right) = 2\sum_{x,y \in V(G)} j_q(x,y) - \sum_{x \in V(G)} j_q(x,x) \, .
\end{equation}
The third and fourth terms in \eqref{eq:almostdone1} are simplified as follows:
\begin{equation} \label{eq:almostdone3}
\begin{aligned}
& \sum_{e =\{u,v\}} \Prm(e) j_q(u,v)
-  \sum_{e = \{u,v\}} \frac{j_q(u,u)^2 +j_q(v,v)^2 -2j_q(u,v)^2}{\ell(e)} = \\
&= \!\!\! \sum_{e =\{u,v\}} \!\! \left( \frac{j_q(u,u)+j_q(v,v)-2j_q(u,v)}{\ell(e)}j_q(u,v) -\frac{j_q(u,u)^2 +j_q(v,v)^2 -2j_q(u,v)^2}{\ell(e)}  \right)\\
&= - \sum_{e =\{u,v\}}\left( j_q(u,u) \frac{j_q(u,u)- j_q(u,v)}{\ell(e)} + j_q(v,v) \frac{j_q(v,v)- j_q(u,v)}{\ell(e)}\right)\\
&= -\sum_{x \in V(G)} j_q(x,x) \sum_{f = \{x, y\}} \frac{j_q(x,x) -j_q(x,y)}{\ell(f)} 
= -\sum_{x \in V(G)} j_q(x,x) \Delta(j_q(x, \cdot))(x) \\
&= -\sum_{x \in V(G)} j_q(x,x) \left(\delta_x(x) - \delta_q(x) \right) = -\sum_{x \in V(G)} j_q(x,x)
\, .
\end{aligned}
\end{equation}
For the first equality we used \eqref{eq:jid4}. The third equality is by a (generalized) handshaking lemma. The result now follows by putting together \eqref{eq:almostdone1}, \eqref{eq:almostdone2}, and \eqref{eq:almostdone3}. 
\end{proof}

\subsection{Average of energy levels, a variation} \label{sec:envar}
For our main application, we will need the following slight variation of Theorem~\ref{thm:energylevel}. Let $\pi \colon C_1(G, \RR) \twoheadrightarrow H_1(G, \RR)$ denote the orthogonal projection (as defined in \S\ref{sec:proj}).

\begin{Definition} \label{def:center}
We define the {\em center} of a rooted spanning tree $(T, q)$ to be
\[
\sigma_T \coloneqq 
\frac{1}{2}\sum_{e \in \Tc_q} {\pi(e)}\, .
\] 
\end{Definition}
\begin{Theorem} \label{thm:avgcenters}
We have the equality
\[
\frac{1}{w(G)}\sum_T w(T)[\sigma_T, \sigma_T] 
= 
\frac{1}{4}\sum_{e = \{u,v\}} \left(r(u,v) - \frac{j_{u}( v, q)^2 + j_{v}( u, q)^2}{\ell(e)}\right) \, ,
\]
where the sum on the left is over all spanning trees $T$ of $G$, and the sum on the right is over all edges $e \in E(G)$.
\end{Theorem}
\begin{proof}
Using \eqref{eq:Fef} and Definition~\ref{def:enlev} we compute:
\begin{equation} \label{eq:sigmaen}
\begin{aligned}
&[\sigma_T, \sigma_T] = \frac{1}{4} [\sum_{e \in \Tc_q} {\pi(e)}, \sum_{e \in \Tc_q} {\pi(e)}]= \frac{1}{4} \left( \sum_{e \in \Tc_q} [\pi(e), \pi(e)] + \sum_{\substack{e,f \in \Tc_q\\ e \ne f}} [\pi(e), \pi(f)]\right) \\
&=\frac{1}{4} \left( \! \sum_{e \in T} \F(e) \ell(e) + \!\!\! \sum_{\substack{e,f \in \Tc_q\\ e \ne f}} \F(e,f) \ell(e) \! \right) \! \! =\frac{1}{4} \left(\! \sum_{e \in T} \F(e) \ell(e) - \!\!\! \sum_{\substack{e,f \in \Tc_q\\ e \ne f}} \xi(e^-, e^+, f^-, f^+) \right) \\
&=\frac{1}{4} \left( \sum_{e \in T} \F(e) \ell(e) +\sum_{e \in \Tc_q} r(e^-, e^+) - \!\!\!\sum_{\substack{e,f \in \Tc_q}} \xi(e^-, e^+, f^-, f^+) \right) \\
&=\frac{1}{4} \sum_{\substack{e \in T\\ e=\{u,v\}}} \left(\F(e) \ell(e) +r(u, v)\right) - \frac{1}{4} {\boldsymbol{\varrho}}(T, q) \, .  \\
\end{aligned}
\end{equation}
By Definition~\ref{def:foster} and \eqref{eq:Pe_sum} and by changing the order of summations, we compute
\begin{equation} \label{eq:rsimplify}
\begin{aligned}
\sum_T \frac{w(T)}{w(G)} &\sum_{\substack{e \in T\\ e=\{u,v\}}} \left(\F(e) \ell(e) +r(u, v)\right) = \sum_{e=\{u,v\}}  \left(\F(e) \ell(e) +r(u, v)\right) \sum_{T \ni e} \frac{w(T)}{w(G)}\\
&= \sum_{e=\{u,v\}}  \left(\F(e) \ell(e) +r(u, v)\right)  \Prm(e)=\sum_{e=\{u,v\}} r(u,v) \, .
\end{aligned}
\end{equation}
The result now follows from \eqref{eq:sigmaen}, \eqref{eq:rsimplify}, and Theorem~\ref{thm:energylevel}.
\end{proof}
\section{Combinatorics of Voronoi polytopes}
\label{sec:comb}

Throughout this section we fix a metric graph $\Gamma$ and a model $G$. We are interested in the combinatorics of the lattice $\left(H_1(G, \ZZ), [\cdot , \cdot]\right)$. More specifically, we study the combinatorics of the Voronoi polytopes $\Vor(\lambda)$ (as defined in \S\ref{sec:voronoi}) for the lattice $\left(H_1(G, \ZZ), [\cdot , \cdot]\right)$. Since $\Vor ( \lambda ) = \Vor(0)+ \lambda$ for all $\lambda \in H_1(G, \ZZ)$, it suffices to understand the Voronoi polytope $\Vor(0)$ around the origin.

Let $\Vol\left(\cdot \right)$ denote the volume measure induced by the bilinear form $[ \cdot , \cdot ]$ on $H_1(G, \RR)$. Let $w(G)$ be the weight of $G$ (as in \S\ref{sec:proj}). Put $g=\dim_\RR H_1(G,\RR)$.

\begin{Lemma} \label{lem:volvor}
$\Vol\left(\Vor(0)\right)  = \sqrt{w(G) }$ .
\end{Lemma}

\begin{proof}
The Voronoi polytopes $\{\Vor(0)+ \lambda \colon \lambda \in H_1(G, \ZZ)\}$ induce a periodic polytopal decomposition of $H_1(G, \RR)$. Therefore $\Vor(0)$, up to some identifications on its boundary, gives a fundamental domain for the translation action of $H_1(G, \ZZ)$ on $H_1(G, \RR)$. Therefore $\Vol\left(\Vor(0)\right) = \sqrt{\det(\Gb)}$ where $\Gb$ is any Gram matrix for the lattice $\left(H_1(G, \ZZ), [\cdot , \cdot]\right)$. 

Fix an orientation $\Oc$ on $G$, and fix a spanning tree $T$ of $G$. It is well-known that $\{\circu(T,e) \colon e \in  \Oc \backslash T\}$ is a basis for $H_1(G, \RR)$.  Let $\Cb_T$ denote
the totally unimodular $g \times m$ matrix whose rows correspond to these basis elements. Then a Gram matrix for the lattice $\left(H_1(G, \ZZ), [\cdot , \cdot]\right)$ is $\Gb_T = \Cb_T \D \Cb_T^{\TP}$. The result now follows from a standard application of the Cauchy--Binet formula for determinants.
\end{proof}

\begin{Remark}
A geometric proof of Lemma~\ref{lem:volvor} can be found in \cite[\S5]{abks}. 
\end{Remark}

Let $\pi \colon C_1(G, \RR) \twoheadrightarrow H_1(G, \RR)$ denote the orthogonal projection (as defined in \S\ref{sec:proj}). Each finite collection of $1$-chains $V = \{\vb_1 , \ldots , \vb_k\} \subset C_1(G,\RR)$ generates a {\em zonotope} $\Z(V)$ defined as
\[
\Z(V) = \left\{ \sum_{i=1}^k \alpha_i \vb_i \colon -1 \leq \alpha_i \leq 1 \right\} \subset C_1(G,\RR) \, .
\]

\begin{Proposition} \label{prop:VorZon}
Fix an orientation $\Oc$ on $G$. Then
\[
\Vor(0) = \frac{1}{2} \, \Z \left(\{ \pi(e) \colon e \in \Oc\}\right) \, .
\]

\end{Proposition}
\begin{proof}
This is well-known. To our knowledge this was first proved in \cite[Proposition~5.2]{os}. See also \cite[Theorem 2]{ssv} for a different proof. 
\end{proof}

Fix an orientation $\Oc$ on $G$. Let $T$ be a spanning tree of $G$. 
We define
\[
\C_T \coloneqq \frac{1}{2} \Z \left( \{ \pi(e) \colon e \in \Oc \backslash T \}\right) \, .
\]

Note that $\C_T$ is independent of the choice of $\Oc$.

\begin{Lemma} \phantomsection \label{lem:CT}
\begin{itemize}
\item[]
\item[(a)] $\C_T$ is a $g$-dimensional {\em parallelotope}. Equivalently, 
$\{\pi(e) \colon e \in  \Oc \backslash T\}$
is a basis for $H_1(G, \RR)$.
\item[(b)] $\Vol(\C_T) = w(T)/\sqrt{w(G)}$ . 
\end{itemize}
\end{Lemma}
\begin{proof}
(a) Note that, for $e,e' \in \Oc \backslash T$, we have $[\pi(e) , \circu(T,e')] = \ell(e) \delta_e(e')$. It is well-known that $\{\circu(T,e) \colon e \in  \Oc \backslash T\}$ is a basis for $H_1(G, \RR)$.

(b) This is proved in \cite[Proposition~5.4]{abks}.  
\end{proof}

Now fix $q \in V(G)$. Recall (see \S\ref{sec:weightedgraph}) that the rooted spanning tree $(T, q)$ comes with a preferred orientation $\Tc_q \subseteq \mathbb{E}(G)$ (where edges are oriented {away from $q$} on the spanning tree $T$). Let $\sigma_T \coloneqq \frac{1}{2}\sum_{e \in \Tc_q} {\pi(e)}$ be the center of $(T,q)$ (Definition~\ref{def:center}).
We now state and prove the main result of this section.

\begin{Theorem} \label{thm:vordecomp}
The collection of parallelotopes 
\[\{\sigma_T + \C_T \colon T \text{ is a spanning tree of } G\}\] 
induces a polytopal decomposition of $\Vor(0)$:
\begin{itemize}
\item[(i)] $\Vor(0) = \bigcup_{T} \left(\sigma_T + \C_T \right)$, the union being over all spanning trees $T$ of $G$,
\item[(ii)] If $T \ne T'$ are two spanning trees such that 
\[ 
\Fc \coloneqq 
\left(\sigma_T + \C_T \right) \cap \left(\sigma_{T'} + \C_{T'} \right)\]
 is nonempty, then 
$\mathcal{F} $ is a face of both $\left(\sigma_T + \C_T \right)$ and $\left(\sigma_{T'} + \C_{T'} \right)$.
\end{itemize}
\end{Theorem}
\begin{proof}
It follows from Proposition~\ref{prop:VorZon} that, for all spanning trees $T$ of $G$, we have 
\begin{equation} \label{eq:ingred1}
\sigma_T + \C_T \subseteq \Vor(0) \, .
\end{equation}
Moreover, by Lemma~\ref{lem:volvor} and Lemma~\ref{lem:CT}~(b), we know
\begin{equation} \label{eq:ingred2}
\sum_T{\Vol\left(\sigma_T + \C_T \right)} = \Vol\left(\Vor(0)\right)
\, ,
\end{equation}
the sum being over all spanning trees $T$ of $G$. 

Next we describe how these parallelotopes can intersect each other. Let $T \ne T'$ be two spanning trees. We choose an orientation $\Oc$ on $G$ that agrees with $\Tc_q$ for $e\in T$, disagrees with $\Tc'_q$ for $e \in T' \backslash T$, and is arbitrary outside $T \cup T'$. We may partition $\Oc$ into the following (possibly empty) subsets: 
\[
\begin{aligned}
\Sc_1 \coloneqq \Tc_q \, \backslash \left( \Tc'_q \cup \overline{\Tc'_q}\right) \,\, , \,\,
\Sc_2 &\coloneqq \overline{\Tc'_q} \, \backslash \left( \Tc_q \cup \overline{\Tc_q}\right)  \,\, , \,\,
\Sc_3 \coloneqq \Tc_q \, \cap \overline{\Tc'_q} \,\, , \,\,
\Sc_4 \coloneqq \Tc_q \, \cap \Tc'_q \,\, , \,\, \\
\Sc_5 &\coloneqq \Oc \backslash \left(\Sc_1 \cup \Sc_2 \cup \Sc_3 \cup \Sc_4 \right) \,\, .
\end{aligned}
\]

Consider the oriented subgraph $\Dc \coloneqq \Sc_1 \cup \Sc_2 \cup \Sc_3$. Note that the oriented subgraph $\Dc$ is obtained from the edge set $T\cup T'$ by removing those edges in $T\cap T'$ that have the same orientation in $\Tc_q$ and in $\Tc'_q$, and by orienting the remaining edges according to $\Oc$. It is easy to check that $\Dc$ is {\em totally cyclic} (i.e. every edge belongs to a directed circuit or, equivalently, it has no directed cocircuit). This implies that there exists a vector $\vb \in H_1(G, \RR)$ such that 
\[
\begin{aligned}
[\vb , \pi(e)] &> 0 \quad , \quad \text{if } e \in \Sc_1  \cup \Sc_2 \cup \Sc_3\, ,\\
[\vb , \pi(e)] &= 0 \quad , \quad \text{if } e \in \Sc_4 \cup \Sc_5 \, .
\end{aligned}
\]
To see this, consider the oriented (cographic) hyperplane arrangement with normal vectors $\{\pi(e) \colon e \in \Oc\}$ in the vector space $H_1(G, \RR)$. It is well known (see, e.g., \cite[Lemma~8.2]{grzas}) that there is a one-to-one correspondence between cells of this hyperplane arrangement and totally cyclic subgraphs of $G$. Our desired vector $\vb$ is any vector in the cell corresponding to $\Dc$.

Consider the function $\hf \colon H_1(G, \RR) \rightarrow \RR$ defined by
\[
\hf(z) = [\vb , z  + \sum_{e \in \Sc_2}{\frac{1}{2}\pi(e)} - \sum_{e \in \Sc_1 \cup \Sc_4}{\frac{1}{2}\pi(e)}] \, .
\]
We compute 
\[
\begin{aligned}
\hf\left( \sigma_T + \C_T \right) &= \hf\left( \frac{1}{2} \, \Z \left(\{ \pi(e) \colon e \in \Sc_2 \cup \Sc_5\}\right) + \frac{1}{2}\sum_{e \in \Sc_1 \cup \Sc_3 \cup \Sc_4} \pi(e)  \right) \\
&= \{ 0 \leq x \leq \sum_{e \in \Sc_2} [\vb,  \pi(e) ] \} + \frac{1}{2}\sum_{e \in \Sc_3} [\vb,  \pi(e) ] \subseteq \RR_{\geq 0}
\, .
\end{aligned}
\]
Similarly, we compute 
\[
\begin{aligned}
\hf\left( \sigma_{T'} + \C_{T'} \right) &= \hf\left( \frac{1}{2} \, \Z \left(\{ \pi(e) \colon e \in \Sc_1 \cup \Sc_5\}\right) -  \frac{1}{2}\sum_{e \in \Sc_2 \cup \Sc_3}  \pi(e) + \frac{1}{2}\sum_{e \in  \Sc_4}  \pi(e) \right) \\
&= \{ - \sum_{e \in \Sc_1} [\vb,  \pi(e) ] \leq x \leq 0\}  - \frac{1}{2}\sum_{e \in \Sc_3} [\vb,  \pi(e) ] \subseteq \RR_{\leq 0}
\, .
\end{aligned}
\]
It follows that the relative interiors of $\sigma_T + \C_T$ and $\sigma_{T'} + \C_{T'}$ are disjoint. Moreover, if $\Fc=
\left(\sigma_T + \C_T \right) \cap \left(\sigma_{T'} + \C_{T'} \right) \ne \emptyset$, then the intersection is contained in $\hf^{-1}(0)$, which is a hyperplane in $H_1(G, \RR)$.
The result follows from this observation, together with \eqref{eq:ingred1} and \eqref{eq:ingred2}. 
\end{proof}

\section{Tropical moments of tropical Jacobians}
\label{sec:tropmoment}

In this section, we prove our promised potential theoretic expression for tropical moments of tropical Jacobians.  

\begin{Theorem} \label{thm:momentformula}
Let $\Gamma$ be a metric graph. Fix a model $G$ of $\Gamma$ and fix a point $q \in V(G)$. Then
\[
I(\Jac(\Gamma)) = \frac{1}{12}\sum_{e} {\F(e)^2 \ell(e) +  \frac{1}{4}\sum_{e = \{u,v\}} \left(r(u,v) - \frac{j_{u}( v, q)^2 + j_{v}( u, q)^2}{\ell(e)}\right)} \, ,
\]
where the sums are over all edges $e \in E(G)$.
\end{Theorem}
\begin{proof}
Recall 
\[ 
I(\Jac(\Gamma)) = \int_{\Vor(0)} [ z,z ] \, \d \mu_L(z) \, ,  \]
where $\mu_L$ is the Lebesgue measure on $H_1(G, \RR)$, normalized to have $\mu_L(\Vor(0)) = 1$. By Theorem~\ref{thm:vordecomp}, we obtain
\begin{equation} \label{eq:Isumint}
I(\Jac(\Gamma)) = \sum_T \int_{\sigma_T + \C_T} [ z,z ] \, \d \mu_L(z) \, ,
\end{equation}
the sums being over all spanning trees $T$ of $G$. We also have
\begin{equation} \label{eq:int_CT}
\begin{aligned}
\int_{\sigma_T + \C_T} [ z,z ] \, \d \mu_L(z) &= \int_{\C_T} [ y+\sigma_T, y+\sigma_T ] \, \d \mu_L(y)\\
&=  \int_{\C_T} \left( [ y, y]+2[\sigma_T, y]+[\sigma_T, \sigma_T] \right) \, \d \mu_L(y)\\
&=  \int_{\C_T} [ y, y] \, \d \mu_L(y) + \int_{\C_T} [\sigma_T, \sigma_T] \, \d \mu_L(y)\, .
\end{aligned}
\end{equation}
The last equality is because ${\C_T}$ is centrally symmetric and $[\sigma_T, y]$ is odd. 

By Lemma~\ref{lem:volvor} and Lemma~\ref{lem:CT}~(b) we have 
\begin{equation} \label{eq:int_center}
\int_{\C_T} [\sigma_T, \sigma_T] \, \d \mu_L(y) = \frac{\Vol(\C_T)}{\Vol\left(\Vor(0)\right)} [\sigma_T, \sigma_T] = \frac{w(T)}{w(G)}[\sigma_T, \sigma_T]\, . 
\end{equation}
On the other hand, 
\begin{equation} \label{eq:int_cell}
\begin{aligned}
\int_{\C_T} [ y, y] \, \d \mu_L(y) &= \frac{\Vol(\C_T)}{\Vol\left(\Vor(0)\right)} \int_{[-\frac{1}{2}, \frac{1}{2}]^g} [\sum_{e \not\in T} \alpha_e \pi(e) , \sum_{e \not\in T} \alpha_e \pi(e)] \, \d \boldsymbol{\alpha} \\
&= \frac{w(T)}{w(G)}  \int_{[-\frac{1}{2}, \frac{1}{2}]^g} \left( \sum_{\substack{e,f \not\in T\\ e \ne f}} \alpha_e  \alpha_{f} [\pi(e), \pi(f)] +  \sum_{e\not \in T} \alpha_e^2   [\pi(e), \pi(e)] \right) \! \d\boldsymbol{\alpha} \\
&= \frac{w(T)}{w(G)} \sum_{e \not \in T} [\pi(e), \pi(e)] \int_{[-\frac{1}{2}, \frac{1}{2}]^g}  \alpha_e^2 \d\boldsymbol{\alpha} \\
\end{aligned}
\end{equation}
Here $\d \boldsymbol{\alpha} = \prod_{e \not\in T}\d\alpha_e$ denotes the usual Lebesgue measure on $\RR^g$. The first equality is by the {\em change of variables} theorem, and the last equality holds because $\alpha_e  \alpha_{f}$ is an odd function on the symmetric domain $[-\frac{1}{2}, \frac{1}{2}]^g$. Clearly (see \S\ref{sec:proj}):
\begin{equation} \label{eq:inteasy}
[\pi(e), \pi(e)] = \F(e) \ell(e) \quad , \quad \int_{[-\frac{1}{2}, \frac{1}{2}]^g}  \alpha_e^2 \d\boldsymbol{\alpha} =
\int_{-\frac{1}{2}}^{\frac{1}{2}}  \alpha_e^2 \d\alpha_e = \frac{1}{12} \,.
\end{equation}
Putting \eqref{eq:Isumint}, \eqref{eq:int_CT}, \eqref{eq:int_center}, \eqref{eq:int_cell}, and \eqref{eq:inteasy} together, we obtain:
\begin{equation} \label{eq:Iweightedsum}
I(\Jac(\Gamma)) = \frac{1}{12} \sum_T   \frac{w(T)}{w(G)} \sum_{e \not \in T} \F(e) \ell(e) + \sum_T \frac{w(T)}{w(G)}[\sigma_T, \sigma_T] \,.
\end{equation}
The first sum is simplified immediately, after changing the order of summations (see Remark~\ref{rmk:foster}~(i)):
\begin{equation} \label{eq:1st_sum}
\sum_T   \frac{w(T)}{w(G)} \sum_{e \not \in T} \F(e) \ell(e) = \sum_{e \in E(G)} \F(e) \ell(e)  \sum_{T \not \ni e} \frac{w(T)}{w(G)} = \sum_{e \in E(G)} \F(e)^2 \ell(e) \, .
\end{equation}
The result now follows from \eqref{eq:Iweightedsum}, \eqref{eq:1st_sum}, and Theorem~\ref{thm:avgcenters}. 
\end{proof}

\begin{Remark}\phantomsection \label{rmk:I}
\begin{itemize}
\item[]
\item[(i)] If $e = \{u,v\}$ is a bridge then one observes that $r(u, v)= \ell(e)$ and 
$\{j_{u}( v, q),  j_{v}( u, q)\} = \{0,\ell(e)\}$ and $\F(e)= 0$. So $e$ contributes $0$ to $I(\Jac(\Gamma))$. This is expected, as bridges do not contribute to the first homology.
\item[(ii)]  By \eqref{eq:jid2}, \eqref{eq:jid3}, and \eqref{eq:jid4} it is clear that everything in the formula for $I(\Jac(\Gamma))$ in Theorem~\ref{thm:momentformula} can be expressed in terms of the entries of $\L_q$ and edge lengths. It follows using Remark \ref{rmk:computej} that computing $I(\Jac(\Gamma))$ takes time at most $O(n^\omega)$, where $n$ is the number of vertices in a model of $\Gamma$, and $\omega$ is the exponent for matrix multiplication.
\end{itemize}
\end{Remark}

\begin{Example}\label{ex:I}

Consider the metric graph $\Gamma$ in Figure~\ref{fig:banana}. 
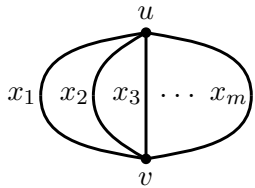
\begin{figure}[h!]
$$
\begin{xy}
(0,0)*+{
	\scalebox{.7}{$
	\begin{tikzpicture}
	\draw[black, ultra thick] (0,1.2) to [out=-170,in=90] (-2,0);
	\draw[black, ultra thick] (-2,0) to [out=-90,in=170] (0,-1.2);
	\draw[black, ultra thick] (0,1.2) to [out=-150,in=90] (-1,0);
	\draw[black, ultra thick] (-1,0) to [out=-90,in=150] (0,-1.2);
	\draw[black, ultra thick] (0,1.2) to [out=-10,in=90] (2,0);
	\draw[black, ultra thick] (2,0) to [out=-90,in=10] (0,-1.2);
	\draw[black, ultra thick] (0,1.2) to (0,-1.2);	
	\fill[black] (0,1.2) circle (.1);
	\fill[black] (0,-1.2) circle (.1);
	\end{tikzpicture}
	$}
};
(-16.5,0)*+{\mbox{{\smaller $x_1$}}};
(-9.5,0)*+{\mbox{{\smaller $x_2$}}};
(-2.5,0)*+{\mbox{{\smaller $x_3$}}};
(3.5,0)*+{\mbox{{ $\ldots$}}};
(11,0)*+{\mbox{{\smaller $x_m$}}};
(0,11)*+{\mbox{{\smaller $u$}}};
(0,-11)*+{\mbox{{\smaller $v$}}};
\end{xy}
$$
\caption{A banana graph $\Gamma$, consisting of $2$ branch points and $m$ edges $\{e_1, \ldots , e_m\}$ with $\ell(e_i) = x_i$.}
\label{fig:banana}
\end{figure} 

Fix the minimal model $G$ whose vertex set consists of the two branch points $u$ and $v$, and let $q=v$. Let $\Rrm$ denote the effective resistance between $u$ and $v$:
\[
{\Rrm}^{-1} = \sum_{i=1}^m {x_i}^{-1} \,.
\]
Then, for each edge $e_i = \{u,v\}$, we have
\[
\F(e_i) = 1 - {\Rrm}/{x_i} \quad , \quad r(u,v) = \Rrm \quad , \quad j_u(v,q) = \Rrm \quad , \quad  j_v(u,q) = 0\,.
\]
By Theorem~\ref{thm:momentformula} one computes:
\[
I(\Jac(\Gamma)) = \frac{1}{12} \left( \sum_{i=1}^m x_i +\frac{m-2}{\sum_{i=1}^m {x_i}^{-1}}\right) 
 \,.
\]
In the special case $x_1 = \ldots = x_m = 1$, the lattice $\left(H_1(\Gamma, \ZZ), [\cdot , \cdot] \right)$ corresponds to the root lattice $A_{m-1}$. So one recovers the formula in \cite[page 460]{cs}:
\[ I(A_{m-1}) = \frac{m^2+m-2}{12m}\, . \]

We record two other special cases for our application in \S\ref{sec:arakelov}:
\begin{itemize}
\item If $m=2$ then the metric graph $\Gamma$ is just a circle, and 
\[I(\Jac(\Gamma)) = \frac{1}{12} \ell \, ,\]
 where $\ell$ is the total length of the circle.
\item If $m=3$ then we have 
\[
I(\Jac(\Gamma)) = \frac{1}{12}\left( x_1+x_2+x_3 + \frac{x_1x_2x_3}{x_1x_2+x_2x_3+x_3x_1}\right) \, .
\]
\end{itemize}
\end{Example}

\section{Heights of Jacobians in dimension two} \label{sec:arakelov} 

In this section we specialize (\ref{two_heights}) to the case where the principally polarized abelian variety $(A,\lambda)$ is a Jacobian variety of dimension two. We use (\ref{two_heights}) together with Example \ref{ex:I} to  give a conceptual explanation of a result due to Autissier \cite{aut}. For more background and for terminology used in this section we refer to \cite{FR1}.

Let $k$ be a number field. Let $X$ be a smooth projective geometrically connected curve of genus two with semistable reduction over $k$. Let $\Jac(X)$ be the Jacobian variety of $X$. As in \S\ref{sec:connection} we denote by $M(k)_0$ and $M(k)_\infty$ the set of non-archimedean places and the set of complex embeddings of $k$. For each $v \in M(k)_0$ we have a metric graph $\Gamma_v$ canonically associated to $X$ at $v$ by taking the dual graph of the geometric special fiber of the stable model of $X$ at $v$, see for example \cite{Zhang93}. The tropical Jacobian $\Jac(\Gamma_v)$  is canonically isometric with the canonical skeleton of the Berkovich analytification of $\Jac(X)$ at $v$. 

Equation (\ref{two_heights}) specializes into the identity
\begin{equation} \label{eqn:specialized} \begin{split} \h_F(\Jac(X)) =  4 \, \h'_L  (\Theta) - 2 \, \kappa_0  +\frac{1}{[k:\QQ]} & \sum_{v \in M(k)_0} I(\Jac(\Gamma_v)) \log Nv \\ &  + \frac{2}{[k:\QQ]} \sum_{v \in M(k)_\infty}  I(\Jac(X_v)) \, . \end{split} 
\end{equation}
In \cite[Th\'eor\`eme~5.1]{aut} Autissier proves an identity
\begin{equation} \label{aut_genus_two} \begin{split} \h_F(\Jac(X)) =  4 \, \h'_L  (\Theta) - 2 \, \kappa_0  +\frac{1}{[k:\QQ]} & \sum_{v \in M(k)_0} \alpha_v \log Nv \\ &  + \frac{2}{[k:\QQ]} \sum_{v \in M(k)_\infty}  I(\Jac(X_v)) \, , \end{split} 
\end{equation}
where for each $v \in M(k)_0$ the local invariant $\alpha_v$ is given explicitly in terms of the metric graph $\Gamma_v$ by means of a table. See the Remarque following the proof of  \cite[Th\'eor\`eme~5.1]{aut} where, for each of the seven possible topological types of stable dual graphs in genus two, the invariant $\alpha_v$ is given in terms of the edge lengths of a minimal model of $\Gamma_v$.  

A simple explicit calculation using Example \ref{ex:I} shows that for all seven topological types as listed in Autissier's table, the equality $\alpha_v = I(\Jac(\Gamma_v))$ is verified. For example, consider case VII in the last row of Autissier's table, corresponding to a banana graph with three edges. The table in this case gives 
\[ \alpha_v = \frac{1}{12}\left( x_1+x_2+x_3 + \frac{x_1x_2x_3}{x_1x_2+x_2x_3+x_3x_1}\right) \, , \]
where $x_1, x_2, x_3$ are the three edge lengths. By Example~\ref{ex:I}, with $m=3$, we see immediately that $\alpha_v = I(\Jac(\Gamma_v))$ for case VII. The cases I-VI are very similar, and in fact simpler, as the irreducible components of the corresponding graphs are either bridges, which by Remark~\ref{rmk:I}~(ii) contribute zero to $I(\Jac( \Gamma_v))$, or circles, which by Example~\ref{ex:I} with $m=2$ contribute ${1}/{12}$ of their total length to  $I(\Jac(\Gamma_v))$. We thus have a complete conceptual explanation of all entries in Autissier's table. 

\section{Tau invariant of metric graphs}
\label{sec:tau}

Let $\Gamma$ be a metric graph. The notion of Arakelov--Green's function $g_\mu(x, y)$ associated to a measure $\mu$ on $\Gamma$ is introduced in \cite{cr,Zhang93}.  It can be shown (\cite[Theorem 2.11]{cr}) that there exists a unique measure $\mu_{\can}$ on $\Gamma$ having total mass $1$, such that $g_{\mu_{\can}}(x, x)$ is a constant. This constant is by definition $\tau(\Gamma)$. Alternatively $\tau(\Gamma)$ can be interpreted as a certain `capacity', with equilibrium measure $\mu_{\can}$ and with potential kernel  ($1/2$ times) the effective resistance function $r(x,y)$ (\cite[Corollary 14.2]{br}). See also \cite{bffourier, br, ciop} for more background, examples, and formulas. 

We will work with the following  definition (see, e.g., \cite[Lemma 14.4]{br}) in terms of the effective resistance function.
\begin{Definition}\label{def:tau}
Fix a point $q \in \Gamma$. Let $f(x) = \frac{1}{2} r(x, q)$. We put
\[
\tau(\Gamma) \coloneqq \int_{\Gamma}{\left(f'(x)\right)^2 \d x} \, ,
\]
where $\d x$ denotes the (piecewise) Lebesgue measure on $\Gamma$. 
\end{Definition}

It is elementary and well-known that the real number $\tau(\Gamma)$ is independent of the choice of $q \in \Gamma$.

We have the following explicit (and efficiently computable) formula for $\tau(\Gamma)$. See also \cite[Proposition 2.9]{ciop} for an equivalent form of this formula. Our proof is somewhat different and avoids `circuit reduction theory'. 

\begin{Theorem} \label{thm:tauformula}
Let $\Gamma$ be a metric graph. Fix a model $G$ of $\Gamma$, and fix a point $q \in V(G)$. Then
\[
\tau(\Gamma) = \frac{1}{12} \sum_{e}\F(e)^2 \ell(e) +\frac{1}{4} \sum_{e= \{u,v\}}  \frac{\left(r(u, q) - r(v, q)\right)^2}{\ell(e)}  \, ,
\]
where the sums are over all edges $e \in E(G)$.
\end{Theorem}

\begin{proof}
Let $f(x)$ be as in Definition~\ref{def:tau}. Let $G$ be a model of $\Gamma$. Then 
\[
\tau(\Gamma) = \sum_{e \in E(G)} \int_e {\left(f'(x)\right)^2 \d x} \, .
\]
We identify each edge segment $e$ in $E(G)$ with an interval $[0,\ell(e)]$ and represent a point $x \in e$ by $x \in [0,\ell(e)]$. By Example~\ref{ex:eff}, we have
\[
2f'(x) =  \F(e) -\frac{r(u,q) - r(v,q)}{\ell(e)} - 2 \frac{\F(e)}{\ell(e)} x \, .
\]
With a direct computation, we obtain
\[
\int_e {\left(f'(x)\right)^2 \d x} = \int_0^{\ell(e)} {\left(f'(x)\right)^2 \d x} = \frac{1}{12} \F(e)^2 \ell(e) + \frac{1}{4} \frac{\left(r(u,q) - r(v,q) \right)^2}{\ell(e)}
\]
and the result follows. 
\end{proof}

The reader will notice the similarity between the right hand side in Theorem~\ref{thm:tauformula} and the right hand side in Theorem~\ref{thm:momentformula}. In fact we can now prove a simple linear relation between $\tau(\Gamma)$, the tropical moment $I(\Jac(\Gamma))$ and the {\em total length} of $\Gamma$. This will be the subject of the final section. 

\section{The linear relation}
\label{sec:linearrel}
\begin{Definition} \label{def:total_length}
Let $\Gamma$ be a metric graph, and fix a model $G$ of $\Gamma$. The {\em total length} of $\Gamma$ is defined by 
\[\ell(\Gamma) \coloneqq \sum_{e \in E(G)}{\ell(e)}\, .\] 
It is easily seen that $\ell(\Gamma)$ is independent of the choice of the model $G$. 
\end{Definition}

\begin{Theorem} \label{thm:linear}
Let $\Gamma$ be a metric graph. The identity
\[ \frac{1}{2}\tau(\Gamma)+I(\Jac(\Gamma)) = \frac{1}{8} \ell(\Gamma) \]
holds in $\RR$.
\end{Theorem}
\begin{proof} Fix a model $G$ of $\Gamma$.
By Theorem~\ref{thm:tauformula} and \eqref{eq:jid1} we have:
\begin{equation} \label{eq:tau2}
\tau(\Gamma) = \frac{1}{12} \sum_{e}\F(e)^2 \ell(e) +\frac{1}{4} \sum_{e= \{u,v\}} \frac{\left(j_{u}(v , q) - j_{v}(u , q)\right)^2}{\ell(e)}  \, ,
\end{equation}
where the sums are over all edges $e \in E(G)$. The result follows from \eqref{eq:tau2}, \eqref{eq:jid3}, Theorem~\ref{thm:momentformula}, and the following direct computation:
\[ 
\begin{aligned}
&\frac{1}{2}\tau(\Gamma)+I(\Jac(\Gamma)) = \\
&=\! \frac{1}{8} \! \sum_{e = \{u,v\}} 
\left( \F(e)^2 \ell(e) + \frac{\left(j_{u}(v , q) - j_{v}(u , q)\right)^2}{\ell(e)} + 2r(u,v) - 2\frac{j_{u}( v, q)^2 + j_{v}( u, q)^2}{\ell(e)} \!\right)\\
&=\! \frac{1}{8} \! \sum_{e = \{u,v\}}
\left( \F(e)^2 \ell(e)  + 2r(u,v) - \frac{\left(j_{u}(v , q) + j_{v}(u , q)\right)^2}{\ell(e)}\right)\\
&=\! \frac{1}{8} \! \sum_{e = \{u,v\}}
\left( \left(1 -\frac{r(u,v)}{\ell(e)}\right)^2 \ell(e)  + 2r(u,v) - \frac{r(u,v)^2}{\ell(e)}\right)\\
&=\! \frac{1}{8}  \! \sum_{e = \{u,v\}} \ell(e) = \frac{1}{8}\, \ell(\Gamma)
 \, .
\end{aligned}
\]
\end{proof}

\input{moment.bbl}


\end{document}